\documentclass{amsart}
\issueinfo{00}
{0} 
{0} 
{2007} 

\newtheorem{theorem}{Theorem}[section]

\newtheorem{corollary}[theorem]{Corollary}
\newtheorem{definition}[theorem]{Definition}
\newtheorem{example}[theorem]{Example}
\newtheorem{lemma}[theorem]{Lemma}
\newtheorem{proposition}[theorem]{Proposition}
\theoremstyle{remark}
\newtheorem{remark}[theorem]{Remark}
\numberwithin{equation}{section}

\begin{document}
\title[Geometric structures associated with a contact metric $(\kappa,\mu)$-space]{Geometric structures associated with\\ a contact metric $(\kappa,\mu)$-space}

\author[B. Cappelletti Montano]{Beniamino Cappelletti Montano}
\address{Dipartimento di Matematica,
Universit\`{a} degli Studi di Bari ``A. Moro'', Via E. Orabona 4,
70125 Bari, Italy} \email{b.cappellettimontano@gmail.com}
\author[L. Di Terlizzi]{Luigia Di Terlizzi}
\address{Dipartimento di Matematica,
Universit\`{a} degli Studi di Bari ``A. Moro'', Via E. Orabona 4,
70125 Bari, Italy} \email{terlizzi@dm.uniba.it}



\begin{abstract}
We prove that any contact metric $(\kappa,\mu)$-space
$(M,\xi,\varphi,\eta,g)$ admits a canonical paracontact metric
structure which is compatible with the contact form $\eta$. We study
such canonical paracontact structure,  proving that it verifies a
nullity condition and induces on the underlying contact manifold
$(M,\eta)$ a sequence of compatible contact and paracontact metric
structures verifying nullity conditions. The behavior of that
sequence, related to the Boeckx invariant $I_M$ and to the
bi-Legendrian structure of $(M,\xi,\varphi,\eta,g)$, is then
studied. Finally we are able to define a canonical Sasakian
structure on any contact metric $(\kappa,\mu)$-space whose Boexkx
invariant satisfies $|I_M|>1$.
\end{abstract}

\maketitle

\section{Introduction}
A contact metric $(\kappa,\mu)$-space is a contact metric manifold
$(M,\varphi,\xi,\eta,g)$ such that the Reeb vector field belongs to
the so-called ``$(\kappa,\mu)$-nullity distribution'', i.e.
satisfies the following condition
\begin{equation}\label{definizione}
R_{X
Y}\xi=\kappa\left(\eta\left(Y\right)X-\eta\left(X\right)Y\right)+\mu\left(\eta\left(Y\right)hX-\eta\left(X\right)hY\right),
\end{equation}
for some real numbers $\kappa$, $\mu$ and for  any
$X,Y\in\Gamma(TM)$; here $R$ denotes the curvature tensor field of
the Levi Civita connection and $2h$ the Lie derivative of the
structure tensor $\varphi$ in the direction of the Reeb vector field
$\xi$. This definition was introduced by Blair, Kouforgiorgos and
Papantoniou in \cite{BKP-95}, as a generalization both of the
Sasakian condition $R_{X
Y}\xi=\eta\left(Y\right)X-\eta\left(X\right)Y$ and of those contact
metric manifolds verifying $R_{X Y}\xi=0$ which were studied by
Blair in \cite{blair-1}.

Recently contact metric $\left(\kappa,\mu\right)$-spaces have
attracted the attention of many authors and various papers have
appeared on this topic (e.g. \cite{Boeckx-08},
\cite{Mino-Luigia-Tripathi-08}, \cite{koufogiorgos}). In fact there
are many motivations for studying
$\left(\kappa,\mu\right)$-manifolds: the first is that, in the
non-Sasakian case (that is for $\kappa\neq 1$), the condition
\eqref{definizione} determines the curvature completely; moreover,
while the values of $\kappa$ and $\mu$ may change, the form of
\eqref{definizione} is invariant under $\mathcal D$-homothetic
deformations; finally, there are non-trivial examples of such
manifolds, the most important being the unit tangent bundle of a
Riemannian manifold of constant sectional curvature endowed with its
standard contact metric structure.

In (\cite{Boeckx-00}) Boeckx provided a complete (local)
classification of non-Sasakian contact metric $(\kappa,\mu)$-spaces
based on the invariant
\begin{equation*}
I_{M}=\frac{1-\frac{\mu}{2}}{\sqrt{1-\kappa}}.
\end{equation*}
Later on, in the recent paper \cite{Mino-sub}, a geometric
interpretation of such invariant in terms of Legendre foliations has
been given.

In this paper we study mainly those (non-Sasakian) contact metric
$(\kappa,\mu)$-spaces such that $I_{M}\neq\pm 1$, showing how rich
the geometry of this wide class of contact metric
$(\kappa,\mu)$-spaces is. In fact we prove that any such contact
metric $(\kappa,\mu)$-manifold is endowed with a non-flat pair of
bi-Legendrian structures, a  $3$-web structure and a canonical
family of contact and paracontact metric structures satisfying
nullity conditions. Such geometric structures are related to each
other and depend on the sign of the Boeckx invariant $I_M$.

The main part of the article is devoted to the study of the
interplays between the theory of contact metric
$(\kappa,\mu)$-spaces and paracontact geometry. The link is just
given by the theory of bi-Legendrian structures. Indeed, as it is
proven in the recent article \cite{Mino-08}, there is a biunivocal
correspondence between the set of (almost) bi-Legendrian structures
and the set  of  paracontact metric structures on the same contact
manifold $(M,\eta)$. Such bijection maps bi-Legendrian structures
onto integrable paracontact metric structures and flat bi-Legendrian
structures onto para-Sasakian structures. Thus, since any contact
metric $(\kappa,\mu)$-manifold $(M,\varphi,\xi,\eta,g)$ is
canonically endowed with the bi-Legendrian structure given by the
eigendistributions corresponding to the non-zero eigenvalues of the
operator $h$, one can associate to $(M,\varphi,\xi,\eta,g)$ a
 paracontact metric structure $(\tilde\varphi,\xi,\eta,\tilde g)$,  which we prove to be given by
\begin{equation}\label{canonical}
\tilde\varphi:=\frac{1}{2\sqrt{1-\kappa}}{\mathcal L}_{\xi}\varphi,
\ \ \tilde g:=d\eta(\cdot,\tilde\varphi\cdot)+\eta\otimes\eta,
\end{equation}
and which we call the \emph{canonical paracontact metric structure}
of the contact metric $(\kappa,\mu)$-space $(M,\varphi,\xi,\eta,g)$.
We study this paracontact structure and we prove that its curvature
tensor field satisfies the relation
\begin{equation*}
\tilde R_{X
Y}\xi=\tilde\kappa\left(\eta(Y)X-\eta(X)Y\right)+\tilde\mu(\eta(Y)\tilde
h X-\eta(X)\tilde h Y),
\end{equation*}
with $\tilde\kappa=\left(1-\frac{\mu}{2}\right)^2+\kappa-2$ and
$\tilde\mu=2$ and where $\tilde h:=\frac{1}{2}{\mathcal
L}_{\xi}\tilde\varphi$. The next step is the study of the structure
defined by the Lie derivative of $\tilde\varphi$ in the direction of
the Reeb vector field. In fact we prove that, if $|I_M|<1$ the
structure $(\varphi_1,\xi,\eta,g_1)$, given by
\begin{equation*}
\varphi_1:=\frac{1}{2\sqrt{-1-\tilde\kappa}}{\mathcal
L}_{\xi}\tilde\varphi, \ \
g_1:=-d\eta(\cdot,\varphi_1\cdot)+\eta\otimes\eta,
\end{equation*}
is a \emph{contact} metric $(\kappa_1,\mu_1)$-structure on
$(M,\eta)$, where $\kappa_1=\kappa+\left(1-\frac{\mu}{2}\right)^2$
and $\mu_1=2$. Whereas, in the case $|I_M|>1$, the structure
$(\tilde\varphi_1,\xi,\eta,\tilde g_1)$, defined by
\begin{equation*}
\tilde\varphi_1:=\frac{1}{2\sqrt{1+\tilde\kappa}}{\mathcal
L}_{\xi}\tilde\varphi, \ \ \tilde
g_1:=d\eta(\cdot,\tilde\varphi_1\cdot)+\eta\otimes\eta,
\end{equation*}
is a \emph{paracontact} metric
$(\tilde\kappa_1,\tilde\mu_1)$-structures, with
$\tilde\kappa_1=\left(1-\frac{\mu}{2}\right)^2+\kappa-2$ and
$\tilde\mu_1=2$. Furthermore, we prove that it is just the canonical
paracontact structure induced by a suitable contact metric
$(\kappa',\mu')$-structure on $M$. Then we show that this procedure
can be iterated and gives rise to a sequence of contact and
paracontact structures associated with the initial contact metric
$(\kappa,\mu)$-structure $(\varphi,\xi,\eta,g)$. The behavior of
such canonical sequence essentially depends  on the Boeckx invariant
$I_M$ of the contact metric $(\kappa,\mu)$-manifold
$(M,\varphi,\xi,\eta,g)$. If $|I_M|>1$,  the sequence consists only
of paracontact structures, whereas in the case $|I_M|<1$ we have an
alternation of contact and paracontact structures (see Theorem
\ref{successione} for all details). Moreover, all the new contact
metric structures on $M$ obtained in this way are in fact
Tanaka-Webster parallel structures (\cite{Boeckx-08}), i.e. the
Tanaka-Webster connection parallelizes both the Tanaka-Webster
torsion and the Tanaka-Webster curvature.

Thus  we have that in a contact metric $(\kappa,\mu)$-space
$(M,\varphi,\xi,\eta,g)$, the $k$-th Lie derivative ${\mathcal
L}_{\xi}\cdots{\mathcal L}_{\xi}\varphi$ of the structure tensor
$\varphi$ in the direction $\xi$, once suitably normalized, defines
a new contact or paracontact structure, depending on the value of
$I_M$.  This last properties shows  another surprising geometric
feature of the invariant $I_M$, linked to the paracontact geometry
of the contact metric $(\kappa,\mu)$-manifold $M$.

Finally we prove that every contact metric $(\kappa,\mu)$-space such
that $|I_M|>1$ admits a canonical compatible Sasakian structure,
explicitly given by
\begin{equation*}
\bar\varphi_{-}:=-\frac{1}{\sqrt{\left(1-\frac{\mu}{2}\right)^2-\left(1-\kappa\right)}}\left(\left(1-\frac{\mu}{2}\right)\varphi+\varphi
h\right), \ \ \bar
g_{-}:=d\eta(\cdot,\bar\varphi_{-}\cdot)+\eta\otimes\eta
\end{equation*}
in the case $I_M<-1$ and
\begin{equation*}
\bar\varphi_{+}:=\frac{1}{\sqrt{\left(1-\frac{\mu}{2}\right)^2-\left(1-\kappa\right)}}\left(\left(1-\frac{\mu}{2}\right)\varphi+\varphi
h\right), \ \ \bar
g_{+}:=-d\eta(\cdot,\bar\varphi_{+}\cdot)+\eta\otimes\eta
\end{equation*}
in the case $I_M>1$. Such Sasakian structures are related to the
above paracontact structures by the formulas
$\bar\varphi_{-}=\tilde\varphi\circ\tilde\varphi_1$ and
$\bar\varphi_{+}=\tilde\varphi_1\circ\tilde\varphi$. In particular,
$(\bar\varphi_{-},\tilde\varphi,\tilde\varphi_1)$ or
$(\bar\varphi_{+},\tilde\varphi_1,\tilde\varphi)$, according to
$I_M<-1$ or $I_M>1$, respectively,  induce an almost
anti-hypercomplex structure, and hence a $3$-web, on the contact
distribution of $(M,\eta)$.

Therefore it appears that a further geometrical interpretation of
the Boeckx invariant is the fact that any contact metric
$(\kappa,\mu)$-space such that $|I_M|<1$ can admit compatible
Tanaka-Webster parallel structures, whereas any contact metric
$(\kappa,\mu)$-space such that $|I_M|>1$ can admit compatible
Sasakian structures.

\bigskip

All manifolds considered here are assumed to be smooth i.e. of the
class  $\mathcal C^\infty$, and connected; we denote by
$\Gamma(\,\cdot\,)$ the set of all sections of a corresponding
bundle. We use the convention that $2u\wedge v=u\otimes v-v\otimes
u$.

\section{Preliminaries}

\subsection{Contact and paracontact structures}

A \emph{contact manifold} is a $(2n+1)$-dimensional smooth manifold
$M$ which carries a $1$-form $\eta$, called \emph{contact form},
satisfying $\eta\wedge\left(d\eta\right)^n\neq 0$ everywhere on $M$.
It is well known that given $\eta$ there exists a unique vector
field $\xi$, called \emph{Reeb vector field}, such that
$i_{\xi}\eta=1$ and $i_{\xi}d\eta=0$. In the sequel we will denote
by $\mathcal D$ the $2n$-dimensional distribution defined by
$\ker\left(\eta\right)$, called the \emph{contact distribution}. It
is easy to see that the Reeb vector field is an infinitesimal
automorphism with respect to the contact distribution and  the
tangent bundle of $M$ splits as the direct sum $TM=\mathcal
D\oplus\mathbb{R}\xi$.

Given a contact manifold $(M,\eta)$ one can consider two different
geometric structures associated with the contact form $\eta$, namely
a contact metric structure and a paracontact metric structure.

\medskip

In fact it is well known that $(M,\eta)$ admits a Riemannian metric
$g$ and a $(1,1)$-tensor field $\varphi $ such that
\begin{equation}\label{acca}
\varphi ^{2}=-I+\eta \otimes \xi, \ \ d\eta \left(
X,Y\right)=g\left( X,\varphi Y\right), \ \ g(\varphi X,\varphi
Y)=g(X,Y)-\eta (X)\eta (Y)
\end{equation}
for all $X,Y\in \Gamma \left( TM\right)$, from which it follows that
$\varphi\xi=0$, $\eta\circ\varphi=0$ and $\eta=g(\cdot,\xi)$. The
structure $\left(\varphi, \xi, \eta, g\right)$ is called a
\emph{contact metric structure} and the manifold $M$ endowed with
such a structure is said to be a {\em contact metric manifold}.  In
a contact metric manifold $M$, the $\left( 1,1\right) $-tensor field
$h:=\frac 12 \mathcal L_\xi \varphi$ is symmetric and satisfies
\begin{equation}
h\xi =0,\;\; \eta\circ h=0,\;\; h\varphi +\varphi h=0,\;\;\nabla\xi
=-\varphi -\varphi h,\;\;  {\rm tr}(h)={\rm tr}(\varphi h)=0,
\label{eq-contact-7}
\end{equation}
where $\nabla$ is the Levi Civita connection of $(M,g)$. The tensor
field $h$ vanishes identically if and only if the Reeb vector field
is Killing, and in this case the contact metric manifold is said to
be \emph{K-contact}.

In any (almost) contact (metric) manifold one can consider the
tensor field $N_{\varphi}$ defined by
\begin{equation}\label{nijenhuis}
N_{\varphi}(X,Y):=\varphi^2[X,Y]+[\varphi X,\varphi
Y]-\varphi[\varphi X,Y]-\varphi[X,\varphi Y]+2d\eta(X,Y)\xi.
\end{equation}
The tensor field $N_\varphi$ satisfies the following formula, which
will turn out very useful in the sequel,
\begin{equation}\label{formulenijenhuis1}
\varphi N_\varphi(X,Y)+N_\varphi(\varphi X,Y)=2\eta(X)h Y,
\end{equation}
for all $X,Y\in\Gamma(TM)$, from which, in particular, it follows
that
\begin{equation}\label{formulenijenhuis2}
\eta(N_{\varphi}(\varphi X,Y))=0.
\end{equation}
Any contact metric manifold such that $N_{\varphi}$ vanishes
identically is said to be \emph{Sasakian}. In terms of the curvature
tensor field, the Sasakian condition is expressed by the following
relation
\begin{equation}\label{condizionesasaki}
R_{XY}\xi=\eta(Y)X-\eta(X)Y.
\end{equation}
Any Sasakian manifold is \emph{K}-contact and in dimension $3$  the
converse also holds (see \cite{Blair-02} for more details). A
natural generalization of the Sasakian condition
\eqref{condizionesasaki} leads to the notion of ``contact metric
$(\kappa,\mu)$-manifold'' (\cite{BKP-95}). Let
$(M,\varphi,\xi,\eta,g)$ be a contact metric manifold. If  the
curvature tensor field of the Levi Civita connection satisfies
\begin{equation}\label{eq-km}
R_{X
Y}\xi=\kappa\left(\eta\left(Y\right)X-\eta\left(X\right)Y\right)+\mu\left(\eta\left(Y\right)hX-\eta\left(X\right)hY\right),
\end{equation}
for some $\kappa,\mu\in\mathbb{R}$, we say that
$(M,\varphi,\xi,\eta,g)$  is a \emph{contact metric
$(\kappa,\mu)$-manifold} (or that $\xi$ belongs to the
$(\kappa,\mu)$-nullity distribution). This definition was introduced
and deeply studied by Blair, Koufogiorgos and Papantoniou in
\cite{BKP-95}. Among other things, the authors proved the following
results.

\begin{theorem}[\cite{BKP-95}]\label{teoremagreci}
Let $\left(M,\varphi,\xi,\eta,g\right)$ be a contact metric
$(\kappa,\mu)$-manifold. Then necessarily $\kappa\leq 1$. Moreover,
if $\kappa=1$ then $h=0$ and $\left(M,\varphi,\xi,\eta,g\right)$ is
 Sasakian; if $\kappa<1$, the contact metric structure is not
Sasakian and $M$ admits three mutually orthogonal integrable
distributions ${\mathcal D}(0)=\mathbb{R}\xi$, ${\mathcal
D}(\lambda)$ and ${\mathcal D}(-\lambda)$ corresponding to the
eigenspaces of $h$, where $\lambda=\sqrt{1-\kappa}$.
\end{theorem}

\begin{theorem}[\cite{BKP-95}]
Let $(M,\varphi,\xi,\eta,g)$ be a contact metric
$(\kappa,\mu)$-manifold. Then the following relations hold, for any
$X,Y\in\Gamma\left(TM\right)$,
\begin{gather}
(\nabla_X \varphi ) Y = g(X,Y+hY)\xi - \eta (Y)(X+hX),\label{Blair1}\\
(\nabla_X h ) Y = ((1-\kappa)g(X, \phi Y)+g(X, \phi hY))\xi +
 \eta(Y)h(\phi X + \phi h X) - \mu \phi h Y,\label{Blair2}\\
(\nabla_X\varphi
h)Y=\bigl(g(X,hY)-(1-\kappa)g(X,\varphi^2Y)\bigr)\xi +
\eta(Y)\bigl(hX -(1-\kappa)\varphi^2 X \bigr) + \mu \eta (X) hY.
\label{Blair3}
\end{gather}
\end{theorem}

Given a non-Sasakian contact metric $(\kappa,\mu)$-manifold $M$,
Boeckx \cite{Boeckx-00} proved that the number
$I_{M}:=\frac{1-\frac{\mu}{2}}{\sqrt{1-\kappa}}$, is an invariant of
the contact metric $(\kappa,\mu)$-structure, and he proved that two
non-Sasakian contact metric $(\kappa,\mu)$-manifolds
$(M_1,\varphi_1,\xi_1,\eta_1,g_1)$ and
$(M_2,\varphi_2,\xi_2,\eta_2,g_2)$ are locally isometric as contact
metric manifolds if and only if $I_{M_1}=I_{M_2}$. Then the
invariant $I_M$ has been used by Boeckx for providing a full
classification of contact metric $(\kappa,\mu)$-spaces. \  The
standard example of contact metric $(\kappa,\mu)$-manifold is given
by the tangent sphere bundle $T_{1}N$ of a Riemannian manifold of
constant curvature $c$ endowed with its standard contact metric
structure. In this case $\kappa=c(2-c)$, $\mu=-2c$ and
$I_{T_{1}N}=\frac{1+c}{|1-c|}$. Therefore as $c$ varies over the
reals, $I_{T_{1}N}$ takes on every value strictly greater than $-1$.
Moreover one can easily find that $I_{T_{1}N}<1$ if and only if
$c<0$.

\bigskip

On the other hand on a contact manifold $(M,\eta)$ one can consider
also compatible paracontact metric structures. We recall (cf.
\cite{kaneyuki1}) that an \emph{almost paracontact structure} on a
$(2n+1)$-dimensional smooth manifold $M$ is given by a
$(1,1)$-tensor field $\tilde\varphi$, a vector field $\xi$ and a
$1$-form $\eta$ satisfying the following conditions
\begin{enumerate}
  \item[(i)] $\eta(\xi)=1$, \ $\tilde\varphi^2=I-\eta\otimes\xi$,
  \item[(ii)] denoted by $\mathcal D$ the $2n$-dimensional distribution generated by
  $\eta$, the tensor field $\tilde\varphi$ induces an almost paracomplex
  structure on each fibre on $\mathcal D$.
\end{enumerate}
Recall that an almost paracomplex structure on a $2n$-dimensional
smooth manifold is a tensor field $J$ of type $(1,1)$ such that
$J\neq I$, $J^2=I$ and the eigendistributions $T^+, T^-$
corresponding to the eigenvalues $1, -1$ of $J$, respectively, have
 dimension $n$.

As an immediate consequence of the definition one has that
$\tilde\varphi\xi=0$, $\eta\circ\tilde\varphi=0$ and the field of
endomorphisms $\tilde\varphi$ has constant rank $2n$. Any almost
paracontact manifold admits a semi-Riemannian metric $\tilde g$ such
that
\begin{equation}\label{compatibile}
\tilde g(\tilde\varphi X,\tilde\varphi Y)=-\tilde
g(X,Y)+\eta(X)\eta(Y)
\end{equation}
for all  $X,Y\in\Gamma(TM)$. Then $(M,\tilde\varphi,\xi,\eta,\tilde
g)$ is called an \emph{almost paracontact metric manifold}. Notice
that any such semi-Riemannian metric is necessarily of signature
$(n+1,n)$. If in addition $d\eta(X,Y)=\tilde g(X,\tilde\varphi Y)$
for all $X,Y\in\Gamma(TM)$, $(M,\tilde\varphi,\xi,\eta,\tilde g)$ is
said to be a \emph{paracontact metric manifold}. On an almost
paracontact manifold one defines the tensor field
\begin{equation*}
N_{\tilde\varphi}(X,Y):=\tilde\varphi^2[X,Y]+[\tilde\varphi
X,\tilde\varphi Y]-\tilde\varphi[\tilde\varphi
X,Y]-\tilde\varphi[X,\tilde\varphi Y]-2d\eta(X,Y)\xi.
\end{equation*}
If $N_{\tilde\varphi}$ vanishes identically the almost paracontact
manifold in question is said to be \emph{normal}.

Moreover, in a paracontact metric manifold one defines a symmetric,
trace-free operator $\tilde h$ by setting $\tilde
h=\frac{1}{2}{\mathcal L}_{\xi}\tilde\varphi$. One can prove (see
\cite{zamkovoy}) that $\tilde h$ is a symmetric operator which
anti-commutes with $\tilde\varphi$ and satisfies $\tilde h\xi=0$,
$\eta\circ\tilde h=0$ and
$\tilde\nabla\xi=-\tilde\varphi+\tilde\varphi\tilde h$, where
$\tilde\nabla$ denotes the Levi Civita connection of $(M,\tilde g)$.
Furthermore $\tilde h$ vanishes identically if and only if $\xi$ is
a Killing vector field and in this case
$(M,\tilde\varphi,\xi,\eta,\tilde g)$ is called a
\emph{$K$-paracontact manifold}. A normal paracontact metric
manifold is said to be a \emph{para-Sasakian manifold}. Also in this
context the para-Sasakian condition implies the \emph{K}-paracontact
condition and the converse holds in dimension $3$. In terms of the
covariant derivative of $\tilde\varphi$ the para-Sasakian condition
may be expressed by
\begin{equation}\label{condizioneparasasaki}
(\tilde \nabla_{X}\tilde\varphi)Y=-\tilde g(X,Y)\xi+\eta(Y)X.
\end{equation}
On the other hand one can prove (\cite{zamkovoy}) that in any
para-Sasakian manifold
\begin{equation}\label{condizioneparasasaki1}
\tilde R_{XY}\xi=\eta(Y)X-\eta(X)Y,
\end{equation}
but, unlike contact metric structures, the condition
\eqref{condizioneparasasaki1} not necessarily implies that the
manifold is para-Sasakian.

In any paracontact metric manifold Zamkovoy introduced a canonical
connection which  plays the same role in paracontact geometry of the
generalized Tanaka-Webster connection (\cite{tanno}) in a contact
metric manifold. In fact the following result holds.

\begin{theorem}[\cite{zamkovoy}]\label{paratanaka}
On a paracontact metric manifold there exists a unique connection
$\tilde{\nabla}^{pc}$, called the \emph{canonical paracontact
connection}, satisfying the following properties:
\begin{enumerate}
  \item[(i)] $\tilde{\nabla}^{pc}\eta=0$, $\tilde{\nabla}^{pc}\xi=0$,
  $\tilde{\nabla}^{pc}\tilde g=0$,
  \item[(ii)]
  $(\tilde\nabla^{pc}_{X}\tilde\varphi)Y=(\tilde\nabla_{X}\tilde\varphi)Y-\eta(Y)(X-\tilde h X)+\tilde g(X-\tilde h X,Y)\xi$,
  \item[(iii)] $\tilde T^{pc}(\xi,\tilde\varphi Y)=-\tilde\varphi \tilde T^{pc}(\xi,Y)$,
  \item[(iv)] $\tilde T^{pc}(X,Y)=2d\eta(X,Y)\xi$ on ${\mathcal D}=\ker(\eta)$.
\end{enumerate}
The explicit expression of this connection is the following
\begin{equation}\label{paradefinition}
\tilde\nabla^{pc}_{X}Y=\tilde\nabla_{X}Y+\eta(X)\tilde\varphi
Y+\eta(Y)(\tilde\varphi X-\tilde\varphi \tilde h X)+\tilde
g(X-\tilde hX,\tilde\varphi Y)\xi.
\end{equation}
Moreover, the torsion tensor field is given by
\begin{equation}\label{paratorsion}
\tilde T^{pc}(X,Y)=\eta(X)\tilde\varphi \tilde h Y -
\eta(Y)\tilde\varphi \tilde h X + 2g(X,\tilde\varphi Y)\xi.
\end{equation}
\end{theorem}

An almost paracontact structure $(\tilde\varphi,\xi,\eta)$ is said
to be \emph{integrable} (\cite{zamkovoy}) if the almost paracomplex
structure $\tilde\varphi|_{\mathcal D}$ satisfies the condition
$N_{\tilde\varphi}(X,Y)\in\Gamma(\mathbb{R}\xi)$  for all
$X,Y\in\Gamma(\mathcal D)$. This is equivalent to require that the
eigendistributions $T^{\pm}$ of $\tilde\varphi$ satisfy
$[T^{\pm},T^{\pm}]\subset T^{\pm}\oplus\mathbb{R}\xi$. For an
integrable paracontact metric manifold, the canonical paracontact
connection shares many of the properties of the Tanaka-Webster
connection on CR-manifolds. For instance we have the following
result.

\begin{theorem}[\cite{zamkovoy}]\label{integrability}
A paracontact metric manifold $(M,\tilde\varphi,\xi,\eta,\tilde g)$
is integrable if and only if the canonical paracontact connection
parallelizes the structure tensor $\tilde\varphi$.
\end{theorem}

In particular, by Theorem \ref{integrability} and
\eqref{condizioneparasasaki} it follows that any para-Sasakian
manifold is integrable.

\subsection{Bi-Legendrian manifolds}\label{preliminari}
Let $(M,\eta)$ be a $(2n+1)$-dimensional contact manifold. It is
well-known that the contact condition $\eta\wedge(d\eta)^n\neq 0$
geometrically means that the contact distribution $\mathcal D$ is as
far as possible from being integrable integrable. In fact one can
prove that the maximal dimension of an involutive subbundle of
$\mathcal D$ is $n$. Such $n$-dimensional integrable distributions
are called \emph{Legendre foliations} of $(M,\eta)$. More  generally
a \emph{Legendre distribution} on a contact manifold $(M,\eta)$ is
an $n$-dimensional subbundle $L$ of the contact distribution not
necessarily integrable but verifying the weaker condition that
$d\eta\left(X,X'\right)=0$ for all $X,X'\in\Gamma\left(L\right)$.

The theory of Legendre foliations has been extensively investigated
in recent years from various points of views. In particular  Pang
(\cite{pang}) provided a classification of Legendre foliations using
 a bilinear symmetric form $\Pi_{\mathcal F}$ on the tangent bundle
of the foliation ${\mathcal F}$, defined by
\begin{equation*}
\Pi_{\mathcal F}\left(X,X'\right)=-\left({\mathcal L}_{X}{\mathcal
L}_{X'}\eta\right)\left(\xi\right)=2d\eta([\xi,X],X').
\end{equation*}
He called a Legendre foliation \emph{positive (negative) definite},
\emph{non-degenerate}, \emph{degenerate} or \emph{flat} according to
the circumstance that the bilinear form $\Pi_{\mathcal F}$ is
positive (negative) definite, non-degenerate, degenerate or vanishes
identically, respectively.   Then for a non-degenerate Legendre
foliation $\mathcal F$,  Libermann (\cite{libermann}) defined a
linear map $\Lambda_{\mathcal F}:TM\longrightarrow T{\mathcal F}$,
whose kernel is ${T\mathcal F}\oplus\mathbb{R}\xi$, such that
\begin{equation}\label{lambda}
\Pi_{\mathcal F}(\Lambda_{\mathcal F} Z,X)=d\eta(Z,X)
\end{equation}
for any $Z\in\Gamma(TM)$, $X\in\Gamma(T{\mathcal F})$. The operator
$\Lambda_{\mathcal F}$ is surjective and  verifies
$(\Lambda_{\mathcal F})^2=0$,  $\Lambda_{\mathcal
F}[\xi,X]=\frac{1}{2}X$ for all $X\in\Gamma(T{\mathcal F})$.  Then
one can extend $\Pi_{\mathcal F}$ to a symmetric bilinear form on
$TM$ by putting
\begin{equation*}
\overline\Pi_{\mathcal F}(Z,Z'):=\left\{
                                   \begin{array}{ll}
                                     \Pi_{\mathcal F}(Z,Z') & \hbox{if $Z,Z'\in\Gamma(T{\mathcal F})$} \\
                                     \Pi_{\mathcal F}(\Lambda_{\mathcal
F} Z,\Lambda_{\mathcal F} Z'), & \hbox{otherwise.}
                                   \end{array}
                                 \right.
\end{equation*}

If $(M,\eta)$  is endowed with  two  transversal  Legendre
distributions $L_1$ and $L_2$,  we  say  that $(M,\eta,L_1,L_2)$ is
an \emph{almost bi-Legendrian manifold}. Thus, in particular, the
tangent bundle of $M$ splits up as the direct sum $TM=L_1\oplus
L_2\oplus\mathbb{R}\xi$. When both $L_1$ and $L_2$ are integrable we
refer to a \emph{bi-Legendrian manifold}. An (almost) bi-Legendrian
manifold is said to be flat, degenerate or non-degenerate if and
only if both the Legendre distributions are flat, degenerate or
non-degenerate, respectively.   Any contact manifold $(M,\eta)$
endowed with a Legendre distribution $L$ admits a canonical almost
bi-Legendrian structure. Indeed let $(\varphi,\xi,\eta,g)$ be a
compatible contact metric structure. Then  the relation $d\eta(\phi
X,\phi Y)=d\eta(X,Y)$  easily implies that $Q:=\phi L$ is a Legendre
distribution on $M$ which is $g$-orthogonal to $L$. $Q$ is usually
referred as the  \emph{conjugate Legendre distribution} of $L$ and
in general is not involutive, even if $L$ is.

In \cite{Mino-05} the existence of  a canonical connection on an
almost bi-Legendrian manifold has been proven:

\begin{theorem}[\cite{Mino-05}]\label{biconnection}
Let $(M,\eta,L_1,L_2)$ be an almost bi-Legendrian manifold. There
exists a unique linear connection ${\nabla}^{bl}$, called
\emph{bi-Legendrian connection}, satisfying the following
properties:
\begin{enumerate}
  \item[(i)] ${\nabla}^{bl} L_1\subset L_1$, \ ${\nabla}^{bl} L_2\subset L_2$,
  \item[(ii)]  ${\nabla}^{bl}\xi=0$, \ ${\nabla}^{bl} d\eta=0$,
  \item[(iii)] ${T}^{bl}\left(X,Y\right)=2d\eta\left(X,Y\right){\xi}$ \ for all
  $X\in\Gamma(L_1)$, $Y\in\Gamma(L_2)$,\\
${T}^{bl}\left(X,\xi\right)=[\xi,X_{L_1}]_{L_2}+[\xi,X_{L_2}]_{L_1}$
  \ for all   $X\in\Gamma\left(TM\right)$,
\end{enumerate}
where ${T}^{bl}$ denotes the torsion tensor field of ${\nabla}^{bl}$
and $X_{L_1}$ and $X_{L_2}$ the projections of $X$ onto the
subbundles $L_1$ and $L_2$ of $TM$, respectively.
\end{theorem}

The behavior of the bi-Legendrian connection in the case of
conjugate Legendre distributions was considered in \cite{Mino-07},
where the following theorem was proven.

\begin{theorem}[\cite{Mino-07}]\label{lemmarocky}
Let $(M,\varphi,\xi,\eta,g)$ be a contact metric manifold endowed
with a Legendre distribution $L$. Let $Q:=\varphi L$ be the
conjugate Legendre distribution of $L$ and $\nabla^{bl}$ the
bi-Legendrian connection associated with $\left(L,Q\right)$. Then
the following statements are equivalent:
\begin{enumerate}
    \item[(i)] $\nabla^{bl} g=0$.
    \item[(ii)] $\nabla^{bl}\varphi=0$.
    \item[(iii)] $\nabla^{bl}_{X}X'=-\left(\varphi\left[X,\varphi
    X'\right]\right)_L$ for all $X,X'\in\Gamma\left(L\right)$, $\nabla^{bl}_{Y}Y'=-\left(\varphi\left[Y,\varphi
    Y'\right]\right)_Q$ for all $Y,Y'\in\Gamma\left(Q\right)$ and
    the tensor field $h$ maps the subbundle $L$
    onto $L$ and the subbundle $Q$ onto $Q$.
    \item[(iv)] $g$ is a bundle-like metric with respect
both to the distribution $L\oplus \mathbb{R}\xi$ and to the
distribution $Q\oplus \mathbb{R}\xi$.
\end{enumerate}
Furthermore, assuming $L$ and $Q$ integrable, (i)--(iv) are
equivalent to the total geodesicity (with respect to the Levi Civita
connection of $g$) of the Legendre foliations defined by $L$ and
$Q$.
\end{theorem}

\section{The foliated structure of a contact metric
$(\kappa,\mu)$-space}\label{primasezione}

Theorem \ref{teoremagreci} implies that any non-Sasakian contact
metric $(\kappa,\mu)$-manifold is endowed with three mutually
orthogonal involutive distributions ${\mathcal D}({\lambda})$,
${\mathcal D}({-\lambda})$ and ${\mathcal D}(0)=\mathbb{R}\xi$,
corresponding to the eigenspaces $\lambda$, $-\lambda$ and $0$ of
the operator $h$, where $\lambda=\sqrt{1-\kappa}$. In particular, as
pointed out in \cite{Mino-Luigia-07},  (${\mathcal D}({\lambda}),
{\mathcal D}({-\lambda}))$ defines a bi-Legendrian structure on
$(M,\eta)$. We started the study of the bi-Legendrian structure of a
contact metric $(\kappa,\mu)$-manifold in \cite{Mino-Luigia-07},
where the explicit expression of the Pang invariant of each Legendre
foliation ${\mathcal D}(\lambda)$ and ${\mathcal D}(-\lambda)$
\begin{gather}
\Pi_{{\mathcal
D}(\lambda)}=\left(2\sqrt{1-\kappa}-\mu+2\right)g|_{{\mathcal
D}(\lambda)\times{\mathcal D}(\lambda)}\label{invariante1}\\
\Pi_{{\mathcal
D}(-\lambda)}=\left(-2\sqrt{1-\kappa}-\mu+2\right)g|_{{\mathcal
D}(-\lambda)\times{\mathcal D}(-\lambda)}\label{invariante2}
\end{gather}
was found (see also \cite{Mino-sub}). It follows that only one among
the following $5$ cases may occur:
\begin{enumerate}
\item[(I)] both ${\mathcal D}(\lambda)$ and  ${\mathcal
D}(-\lambda)$ are positive definite;
\item[(II)] ${\mathcal D}(\lambda)$ is positive definite and  ${\mathcal
D}(-\lambda)$ is negative definite;
\item[(III)] both ${\mathcal D}(\lambda)$ and  ${\mathcal
D}(-\lambda)$ are negative definite;
\item[(IV)]  ${\mathcal D}(\lambda)$ is positive definite and  ${\mathcal
D}(-\lambda)$ is flat;
\item[(V)] ${\mathcal D}(\lambda)$ is flat and ${\mathcal
D}(-\lambda)$ is negative definite.
\end{enumerate}
Moreover, the bi-Legendrian structure $({\mathcal
D}(\lambda),{\mathcal D}(-\lambda))$ belongs to the class (I), (II),
(III), (IV), (V) if and only if $I_M>1$, $-1<I_M<1$, $I_M<-1$,
$I_M=1$, $I_M=-1$, respectively.

Furthermore, the following characterization of contact metric
$(\kappa,\mu)$-manifolds in terms of Legendre foliations holds.

\begin{theorem}[\cite{Mino-Luigia-07}]\label{principale0}
Let $(M,\varphi,\xi,\eta,g)$ be a non-Sasakian contact metric
manifold. Then $(M,\varphi,\xi,\eta,g)$ is a contact metric
$\left(\kappa,\mu\right)$-manifold if and only if it admits two
mutually orthogonal Legendre distributions $L$ and $Q$ and a unique
linear connection $\bar{\nabla}$ satisfying the following
properties:
\begin{enumerate}
\item[{\rm (i)}] $\bar{\nabla}L\subset L$,\quad $\bar{\nabla} Q\subset Q$,
\item[{\rm (ii)}] $\bar{\nabla}\eta=0$,\quad  $\bar{\nabla}d\eta=0$,\quad  $\bar{\nabla}g=0$,\quad  $\bar{\nabla}\varphi=0$,\quad  $\bar{\nabla}h=0$,
\item[{\rm (iii)}] $\bar{T}\left(X,Y\right)=2d\eta\left(X,Y\right){\xi}$ \quad
for  all $X,Y\in\Gamma({\mathcal{D}})$,\\
$\bar{T}(X,\xi)=[\xi,X_{L}]_{Q}+[\xi,X_{Q}]_{L}$ \quad  for all
$X\in\Gamma(TM)$,
\end{enumerate}
where $\bar{T}$ denotes the torsion tensor field of $\bar{\nabla}$
and $X_L$ and $X_Q$ are, respectively, the projections of $X$ onto
the subbundles $L$ and $Q$ of $TM$. Furthermore, $L$ and $Q$ are
integrable and coincide with the eigenspaces ${\mathcal D}(\lambda)$
and ${\mathcal D}(-\lambda)$ of the operator $h$, and $\bar\nabla$
coincides in fact with the bi-Legendrian connection $\nabla^{bl}$
associated to the bi-Legendrian structure $(L,Q)$.
\end{theorem}

In particular, from \eqref{invariante1}--\eqref{invariante2} it
follows that $\nabla^{bl}\Pi_{{\mathcal
D}(\lambda)}=\nabla^{bl}\Pi_{{\mathcal D}(-\lambda)}=0$. \
Conversely one has the following theorem.

\begin{theorem}[\cite{Mino-sub}]\label{legendre1}
Let $(M,\eta)$ be a contact manifold endowed with a  bi-Legendrian
structure $({\mathcal F}_1,{\mathcal F}_2)$ such that
$\nabla^{bl}\Pi_{{\mathcal F}_1}=\nabla^{bl}\Pi_{{\mathcal F}_2}=0$.
Assume that one of the following conditions holds
\begin{itemize}
  \item[(I)] ${\mathcal F}_1$ and ${\mathcal F}_2$ are positive
definite and there exist two positive
 numbers $a$ and $b$ such that $\overline\Pi_{{\mathcal
F}_1}=ab\overline\Pi_{{\mathcal F}_2}$ on $T{\mathcal F}_1$ and
$\overline\Pi_{{\mathcal F}_2}=ab\overline\Pi_{{\mathcal F}_1}$ on
$T{\mathcal F}_2$,
  \item[(II)] ${\mathcal F}_1$ is positive definite and ${\mathcal F}_2$
is negative definite and there exist $a>0$ and $b<0$ such that
$\overline\Pi_{{\mathcal F}_1}=ab\overline\Pi_{{\mathcal F}_2}$ on
$T{\mathcal F}_1$ and $\overline\Pi_{{\mathcal
F}_2}=ab\overline\Pi_{{\mathcal F}_1}$ on $T{\mathcal F}_2$,
  \item[(III)]  ${\mathcal F}_1$ and ${\mathcal F}_2$ are negative
definite and there exist two negative
 numbers $a$ and $b$ such that $\overline\Pi_{{\mathcal
F}_1}=ab\overline\Pi_{{\mathcal F}_2}$ on $T{\mathcal F}_1$ and
$\overline\Pi_{{\mathcal F}_2}=ab\overline\Pi_{{\mathcal F}_1}$ on
$T{\mathcal F}_2$.
\end{itemize}
Then $(M,\eta)$ admits a  compatible contact metric structure
$(\varphi,\xi,\eta,g)$ such that
\begin{enumerate}
  \item[(i)] if $a=b$, $(M,\varphi,\xi,\eta,g)$ is a Sasakian
manifold;
  \item[(ii)] if $a\neq b$, $(M,\varphi,\xi,\eta,g)$ is a contact metric
$(\kappa,\mu)$-manifold, whose  associated bi-Legendrian structure
is $({\mathcal F}_1,{\mathcal F}_2)$, where
\begin{equation}\label{costanti0}
\kappa=1-\frac{(a-b)^2}{16}, \ \ \ \mu=2-\frac{a+b}{2}.
\end{equation}
\end{enumerate}
\end{theorem}


\section{The canonical paracontact structure of a contact metric $(\kappa,\mu)$-space}

In \cite{Mino-08} the interplays between paracontact geometry and
the theory of bi-Legendrian structures have been studied. More
precisely it has been proven the existence of a biunivocal
correspondence $\Psi:{\mathcal{AB}}\longrightarrow{\mathcal{PM}}$
 between the set ${\mathcal{AB}}$ of  almost bi-Legendrian structures
 and the set of paracontact metric structures $\mathcal{PM}$ on the same contact manifold
$(M,\eta)$. This bijection maps bi-Legendrian structures onto
integrable paracontact structures, flat almost bi-Legendrian
structures onto \emph{K}-paracontact structures and flat
bi-Legendrian structures onto para-Sasakian structures. For the
convenience of the reader we recall more explicitly how the above
biunivocal correspondence is defined. If $(L_1,L_2)$ is an almost
bi-Legendrian structure on $(M,\eta)$, the corresponding paracontact
metric structure $(\tilde\varphi,\xi,\eta,\tilde g)=\Psi(L_1,L_2)$
is given by
\begin{equation}\label{biunivoca}
\tilde\varphi|_{L_1}=I, \ \tilde\varphi|_{L_2}=-I, \
\tilde\varphi\xi=0, \ \ \tilde g:=d\eta(\cdot,\tilde\varphi
\cdot)+\eta\otimes\eta.
\end{equation}
Moreover, the relationship between the bi-Legendrian and the
canonical paracontact connections has been investigated, proving
that in the integrable case they in fact coincide:

\begin{theorem}[\cite{Mino-08}]\label{connection}
Let  $(M,\eta,L_1,L_2)$  be  an  almost  bi-Legendrian  manifold and
 let $(\tilde\varphi,\xi,\eta,\tilde g)=\Psi(L_1,L_2)$ be the paracontact metric
structure induced on $M$ by \eqref{biunivoca}. Let $\nabla^{bl}$ and
$\tilde\nabla^{pc}$ be the corresponding bi-Legendrian and canonical
paracontact connections. Then
\begin{enumerate}
  \item[(a)] $\nabla^{bl}\tilde\varphi=0$, $\nabla^{bl}\tilde g=0$,
  \item[(b)] the bi-Legendrian and the canonical paracontact connections coincide if and only if
  the induced paracontact metric structure is integrable.
\end{enumerate}
\end{theorem}

As we have stressed in $\S$ \ref{primasezione}, any (non-Sasakian)
contact metric $(\kappa,\mu)$-manifold $(M,\varphi,\xi,\eta,g)$
carries a canonical bi-Legendrian structure $({\mathcal
D}(\lambda),{\mathcal D}(-\lambda))$ which, in some sense,
completely characterizes the contact metric $(\kappa,\mu)$-structure
itself. Then we present  the following definition.

\begin{definition}
The paracontact metric structure $(\tilde\varphi,\xi,\eta,\tilde
g):=\Psi({\mathcal D}(\lambda),{\mathcal D}(-\lambda))$ is said to
be the \emph{canonical paracontact metric structure} of the
(non-Sasakian) contact metric $(\kappa,\mu)$-space
$(M,\varphi,\xi,\eta,g)$.
\end{definition}

In this section we deal with the study of the canonical paracontact
metric structure of a contact metric $(\kappa,\mu)$-space. The first
remark is that, since ${\mathcal D}(\lambda)$ and ${\mathcal
D}(-\lambda)$ are involutive, $(\tilde\varphi,\xi,\eta,\tilde g)$ is
integrable so that, by Theorem \ref{connection}, the connection
stated in Theorem \ref{principale0} and the canonical paracontact
connection of $(\tilde\varphi,\xi,\eta,\tilde g)$ coincide.

Now we show a more explicit expression for the canonical
paracontact metric structure which will turn useful in the sequel.

\begin{theorem}\label{principale1}
Let $(M,\varphi,\xi,\eta,g)$ be a non-Sasakian contact metric
$(\kappa,\mu)$-space. Then the canonical paracontact metric
structure $(\tilde\varphi,\xi,\eta,\tilde g)$ of $M$ is given by
\begin{equation}\label{paracanonical}
\tilde\varphi:=\frac{1}{\sqrt{1-\kappa}}h, \ \ \  \tilde
g:=\frac{1}{\sqrt{1-\kappa}}d\eta(\cdot,h\cdot)+\eta\otimes\eta.
\end{equation}
\end{theorem}
\begin{proof}
It is well known that in any contact metric $(\kappa,\mu)$-manifold
one has $h^2=(\kappa-1)\varphi^2$ (\cite{BKP-95}). From this
relation it follows that the tensor field
$\tilde\varphi:=\frac{1}{\sqrt{1-\kappa}}h$ satisfies
$\tilde\varphi^2=\frac{1}{1-\kappa}h^2=-\varphi^2=I-\eta\otimes\xi$.
Moreover, $\tilde\varphi$ induces an almost paracomplex structure on
the subbundle $\mathcal D$, given  by the $n$-dimensional
distributions ${\mathcal D}(\lambda)$ and ${\mathcal D}(-\lambda)$.
Thus $\tilde\varphi$ defines an almost paracontact structure on $M$.
Next, we define a compatible metric $\tilde g$ by setting
\begin{equation}\label{metricacanonica1}
\tilde g(X,Y):=d\eta(X,\tilde\varphi Y)+\eta(X)\eta(Y)
\end{equation}
for all $X,Y\in\Gamma(TM)$.  In fact, by using \eqref{eq-contact-7},
we have, for any $X,Y\in\Gamma(TM)$,
\begin{gather*}
\tilde g(Y,X)=\frac{1}{\sqrt{1-\kappa}}d\eta(Y,hX)+\eta(Y)\eta(X)=\frac{1}{\sqrt{1-\kappa}}g(Y,\varphi hX)+\eta(Y)\eta(X)\\
=\frac{1}{\sqrt{1-\kappa}}g(X,\varphi h
Y)+\eta(X)\eta(Y)=d\eta(X,\tilde\varphi Y)+\eta(X)\eta(Y)=\tilde
g(X,Y),
\end{gather*}
thus $\tilde g$ defines a semi-Riemannian metric. Moreover,
$g(\tilde\varphi X,\tilde\varphi Y)=d\eta(\tilde\varphi
X,Y-\eta(Y)\xi)+\eta(\tilde\varphi X)\eta(\tilde\varphi
Y)=d\eta(\tilde\varphi X,Y)=-\tilde g(X,Y)+\eta(X)\eta(Y)$ and
$g(X,\tilde\varphi Y)=d\eta(X,\tilde\varphi^2
Y)+\eta(X)\eta(\tilde\varphi Y)=d\eta(X,Y-\eta(Y)\xi)=d\eta(X,Y)$
for all $X,Y\in\Gamma(TM)$. Hence $(\tilde\varphi,\xi,\eta,\tilde
g)$ is a paracontact metric structure. Finally,  the paracontact
metric structure defined by \eqref{paracanonical} coincides with the
canonical paracontact metric structure of the contact metric
$(\kappa,\mu)$-space $(M,\varphi,\xi,\eta,g)$ as \eqref{biunivoca}
shows.
\end{proof}

The relationship between the Levi Civita connections of $(M,g)$
and $(M,\tilde g)$ is given in the following proposition.

\begin{proposition}\label{levicivita1}
Under the same hypotheses and notation of Theorem 3.4 we have the
following relationship between the Levi Civita connections
$\nabla$, $\tilde \nabla$ of $g$ and $\tilde g$, respectively:
\begin{align*}
\tilde \nabla_X Y&=\nabla_X Y + \frac \mu 2 \bigl(\eta (X) \varphi Y
+\eta (Y) \varphi X  \bigr) - \frac 1 {\sqrt {1-\kappa}} \bigl(
\eta(X) h Y+\eta(Y) h X   \bigr)\\
 &\quad + \frac 12\left(\frac {2-\mu}
{\sqrt {1-\kappa}} g(h X,Y)-2 {\sqrt{1-\kappa}}
g(\varphi^2X,Y)-2g(X,\varphi Y)+2X(\eta (Y)) - \eta(\nabla_XY)
\right)\xi.
\end{align*}
\end{proposition}
\begin{proof}
By using Theorem 3.4 we get for each $X,Y,Z\in \Gamma (TM)$
\begin{eqnarray*}
2\tilde g(\tilde \nabla_XY,Z)&=& X(\tilde g (Y,Z))+Y(\tilde g(X,Z)) - Z(\tilde g(X,Y))\\&& +\tilde g([X,Y],Z)+\tilde g([Z,X],Y)-\tilde g([Y,Z],X)\\
&=& \frac 1{\sqrt {1-\kappa}}\bigl(  X(g(Y,\varphi h Z))+Y(g(X,\varphi h Z))-Z(g(X,\varphi h Y))\\
&& + g([X,Y],\varphi h Z)+  g([Z,X],\varphi h Y))-g([Y,Z],\varphi h
X)) \bigr)\\  && + X(\eta (Y)\eta(Z)) + Y(\eta (X)\eta (Z)) - Z(\eta
(X)\eta(Y))\\ && + \eta([X,Y])\eta (Z)+ \eta ([Z,X])\eta
(Y)-\eta([Y,Z])\eta(X).
\end{eqnarray*}
Hence if we apply the symmetry of $\varphi \circ h$ and the
parallelism of $g$ with respect to $\nabla$, we  obtain
\begin{align*}
2\tilde g(\tilde \nabla_XY,Z)&=\frac 1 {\sqrt {1-\kappa}}\bigl( 2
g(\varphi h\nabla_XY,Z) + g(Y,(\nabla_X\varphi
h)Z)+g(X,(\nabla_Y\varphi h)Z)-g(X,(\nabla_Z\varphi h)Y) \bigr) \\
&\quad+2\bigl(d\eta(X,Z)\eta (Y) +d\eta(Y,Z)\eta(X)
-d\eta(X,Y)\eta(Z)+X(\eta(Y))\eta(Z)\bigr),
\end{align*}
so that by using \eqref{Blair3}, after a long but straightforward
calculation
\begin{align*}
2\tilde g(\tilde \nabla_XY,Z)&=g\left( \frac
1{\sqrt{1-\kappa}}\bigl( 2 \varphi h (\nabla_XY)+ \mu \bigl(\eta
(X)hY+\eta (Y) hX\bigr)
- 2\bigl(\eta(X)\varphi Y+ \eta (Y)\varphi X \bigr), Z\right)\\
&\quad+2g\left(\frac {2-\mu}{2\sqrt {1-\kappa}} g(hX,Y) -{\sqrt
{1-\kappa}}g(\varphi^2X,Y)-g(X,\varphi Y)+ X(\eta(Y))  \bigr)\xi, Z
 \right).
\end{align*}
It is easy to see that $\tilde g(\tilde \nabla_X Y,\xi)=\eta
(\tilde\nabla_XY)$ and then by the previous identity and Theorem 3.4
we get
\begin{equation}
\label{LC} \varphi h \tilde \nabla_X Y = \varphi h \nabla_XY +\frac
\mu 2\bigl(\eta (X)hY+\eta(Y) hX\bigr) - \sqrt {1-\kappa}(\eta
(X)\varphi Y + \eta (Y) \varphi X ).
\end{equation}
We finally apply $\varphi h$ to both the sides of \eqref{LC}, use $h
\varphi =-\varphi h$, $h^2=(\kappa-1)\varphi^2$ and
straightforwardly get the claimed relation.
\end{proof}

We now prove that the canonical paracontact metric structure
$(\tilde\varphi,\xi,\eta,\tilde g)$ satisfies a suitable nullity
condition. To this end we need to prove the following fundamental
lemmas.

\begin{lemma}\label{lemmaluigia}
For the canonical paracontact metric structure
$(\tilde\varphi,\xi,\eta,\tilde g)$ of Theorem \ref{principale1}, we
have
\begin{equation}\label{formule}
\tilde h=\frac{1}{2\sqrt{1-k}}\left(\left(2-\mu\right)\varphi\circ
h+2\left(1-\kappa\right)\varphi\right), \ \ \ {\tilde
h}^2=\left(1-\kappa-\left(1-\frac{\mu}{2}\right)^2\right)\varphi^2.
\end{equation}
\end{lemma}
\begin{proof}
Using the identities $\nabla\xi=-\varphi-\varphi h$,
$\nabla_{\xi}\varphi=0$ and $\varphi^2 h=-h$, we get
\begin{align*}
2\tilde h&=(\mathcal L_\xi (\mathcal L_\xi \varphi))
X\\
&=[\xi,(\mathcal L_\xi\varphi) X]-(\mathcal L_\xi\varphi)[\xi,X]\\
&=[\xi,[\xi,\varphi X]-2[\xi,\varphi[\xi,X]]+\varphi [\xi,[\xi,X]]\\
&=\nabla_{\xi}[\xi,\varphi X]+\varphi[\xi,\varphi X]+\varphi h [\xi,\varphi X]-2\nabla_{\xi}\varphi[\xi,X]-2(\varphi^2[\xi,X]+\varphi h\varphi[\xi,X])+\varphi\nabla_{\xi}[\xi,X]\\
&\quad-\varphi(-\varphi[\xi,X]-\varphi h[\xi,X])\\
&=\nabla_{\xi}\nabla_{\xi}\varphi X-\nabla_{\xi}(-\varphi^2
X-\varphi h \varphi X)+\varphi\nabla_{\xi}\varphi
X-\varphi(-\varphi^2 X-\varphi h \varphi X)+\varphi h
\nabla_{\xi}\varphi X\\
&\quad-\varphi h (-\varphi^2
X-\varphi h \varphi X)-2\nabla_{\xi}\varphi\nabla_{\xi}X+2\nabla_{\xi}\varphi(-\varphi X-\varphi h X)-2\varphi^2\nabla_{\xi}X\\
&\quad+2\varphi^2(-\varphi X-\varphi h X)+2\varphi^2 h
\nabla_{\xi}X-2\varphi^2 h(-\varphi X - \varphi h
X)+\varphi\nabla_{\xi}\nabla_{\xi}X\\
&\quad-\varphi\nabla_{\xi}(-\varphi X-\varphi h
X)+\varphi^2\nabla_{\xi}X-\varphi^2(-\varphi X-\varphi h
X)+\varphi^2 h\nabla_{\xi}X-\varphi^2 h(-\varphi X-\varphi h X)\\
&=\nabla_{\xi}\varphi^2 X+\nabla_{\xi}h X+\nabla_{\xi}\varphi^2
X-\varphi X-h\varphi X+h\nabla_{\xi}X-\varphi h X+h^2\varphi
X-2\nabla_{\xi}\varphi^2 X\\
&\quad-2\nabla_{\xi}\varphi^2 h X-2\varphi^2\nabla_{\xi}X+2\varphi X+2\varphi h X-2h\nabla_{\xi}X-2h\varphi X+2h^2\varphi X+\varphi^2\nabla_{\xi}X\\
&\quad+\varphi^2\nabla_{\xi}h X+\varphi^2\nabla_{\xi}X-\varphi
X-\varphi h X - h\nabla_{\xi}X-h\varphi X+h^2\varphi X\\
&=2(\nabla_\xi h)X
+4h^2\varphi X -4h \varphi X.
\end{align*}
Now since $h^2=(\kappa-1)\varphi^2$ and  $\nabla_\xi h= \mu h
\varphi$ (\cite{BKP-95}), we obtain the first identity in
\eqref{formule}, while the second is a straightforward consequence.
\end{proof}

\begin{lemma}\label{lemmamino}
Let $(M,\varphi,\xi,\eta,g)$ be a contact metric
$(\kappa,\mu)$-manifold and let $(\tilde\varphi,\xi,\eta,\tilde g)$
be the canonical paracontact metric structure induced on $M$,
according to Theorem \ref{principale1}. Then the Levi Civita
connection $\tilde\nabla$ of $(M,\tilde g)$ verifies
\begin{gather}\label{formula1}
(\tilde\nabla_{X}\tilde\varphi)Y=-\tilde g(X-\tilde h
X,Y)\xi+\eta(Y)(X-\tilde h X),\\
(\tilde\nabla_{X}\tilde h)Y=-\eta(Y)(\tilde\varphi\tilde h
X-\tilde\varphi\tilde h^2 X)-2\eta(X)\tilde\varphi\tilde h Y-\tilde
g(X,\tilde\varphi\tilde h Y+\tilde\varphi\tilde h^2
Y)\xi,\label{formula2}
\end{gather}
for all $X,Y\in\Gamma(TM)$.
\end{lemma}
\begin{proof}
\eqref{formula1} easily follows from the integrability of
$(\tilde\varphi,\xi,\eta,\tilde g)$, taking  Theorem
\ref{integrability} into account. In order to prove
\eqref{formula2}, let $\nabla^{bl}$ be the bi-Legendrian connection
associated to the bi-Legendrian structure $({\mathcal
D}(\lambda),{\mathcal D}(-\lambda))$. Notice that $\nabla^{bl}$
coincides with the canonical paracontact connection
$\tilde\nabla^{pc}$, so that, by using the first formula in
\eqref{formule} and since, by Theorem \ref{principale0},
$\nabla^{bl}h=\nabla^{bl}\varphi=0$, we have
\begin{align}\label{formulapreliminare1}
\nonumber(\tilde\nabla^{pc}_{X}\tilde h)Y&=(\nabla^{bl}_{X}\tilde
h)Y\\
&=\frac{1}{2\sqrt{1-k}}\left((2-\mu)(\nabla^{bl}_{X}\varphi h)Y+2(1-k)(\nabla^{bl}_{X}\varphi)Y\right)\\
&=\frac{2-\mu}{2\sqrt{1-k}}\left((\nabla_{X}^{bl}\varphi)h Y +
\varphi(\nabla_{X}^{bl}h)Y\right)+\frac{1-k}{\sqrt{1-k}}(\nabla^{bl}_{X}\varphi)Y=0.
\nonumber
\end{align}
Now, by \eqref{paradefinition}, \eqref{formulapreliminare1} and the
properties of the operator $\tilde h$,
\begin{align*}
(\tilde\nabla_{X}\tilde h)Y&=\tilde\nabla_{X}\tilde h Y - \tilde h
\tilde\nabla_{X}Y\\
&=(\tilde\nabla^{pc}_{X}\tilde h) Y-\eta(X)\tilde\varphi\tilde h
Y-\eta(\tilde h Y)(\tilde\varphi X-\tilde\varphi\tilde h X)-\tilde
g(X,\tilde\varphi\tilde h Y)\xi+\tilde g(\tilde h
X,\tilde\varphi\tilde h Y)\xi\\
&\quad+\eta(X)\tilde h\tilde\varphi Y+\eta(Y)(\tilde h\tilde\varphi
X-\tilde h\tilde\varphi\tilde h X)+\tilde g(X,\tilde\varphi Y)\tilde
h\xi-\tilde g(\tilde h X,\tilde\varphi Y)\tilde h\xi\\
&=-\eta(Y)(\tilde\varphi\tilde h X-\tilde\varphi\tilde h^2
X)-2\eta(X)\tilde\varphi\tilde h Y-\tilde g(X,\tilde\varphi\tilde h
Y+\tilde\varphi\tilde h^2 Y)\xi,
\end{align*}
as claimed.
\end{proof}

We now are able to prove the following result.

\begin{theorem}\label{principale2}
Let $(M,\varphi,\xi,\eta,g)$ be a contact metric
$(\kappa,\mu)$-manifold and let $(\tilde\varphi,\xi,\eta,\tilde g)$
be the canonical paracontact metric structure induced on $M$. Then
the curvature tensor field of the Levi Civita connection of
$(M,\tilde g)$ verifies the following relation
\begin{equation*}
\tilde R_{X
Y}\xi=\tilde\kappa\left(\eta(Y)X-\eta(X)Y\right)+\tilde\mu(\eta(Y)\tilde
h X-\eta(X)\tilde h Y),
\end{equation*}
for all $X,Y\in\Gamma(TM)$, where
\begin{equation}\label{valori}
\tilde\kappa=\kappa-2+\left(1-\frac{\mu}{2}\right)^2, \ \ \
\tilde\mu=2.
\end{equation}
\end{theorem}
\begin{proof}
First we prove the preliminary formula
\begin{equation}\label{formulapreliminare}
\tilde R_{X
Y}\xi=-(\tilde\nabla_{X}\tilde\varphi)Y+(\tilde\nabla_{Y}\tilde\varphi)X+(\tilde\nabla_{X}\tilde\varphi)\tilde
h Y+\tilde\varphi((\tilde\nabla_{X}\tilde
h)Y)-(\tilde\nabla_{Y}\tilde\varphi)\tilde h
X-\tilde\varphi((\tilde\nabla_{Y}\tilde h)X).
\end{equation}
Indeed for all $X,Y\in\Gamma(TM)$, using the identity
$\tilde\nabla\xi=-\tilde\varphi+\tilde\varphi\tilde h$, we get
\begin{align*}
\tilde R_{X
Y}\xi&=\tilde\nabla_{X}\tilde\nabla_{Y}\xi-\tilde\nabla_{Y}\tilde\nabla_{X}\xi-\tilde\nabla_{[X,Y]}\xi\\
&=-\tilde\nabla_{X}\tilde\varphi
Y+\tilde\nabla_{X}\tilde\varphi\tilde h
Y+\tilde\nabla_{Y}\tilde\varphi
X-\tilde\nabla_{Y}\tilde\varphi\tilde h
X+\tilde\varphi[X,Y]-\tilde\varphi\tilde h[X,Y]\\
&=-\tilde\nabla_{X}\tilde\varphi
Y+\tilde\nabla_{X}\tilde\varphi\tilde h
Y+\tilde\nabla_{Y}\tilde\varphi
X-\tilde\nabla_{Y}\tilde\varphi\tilde h
X+\tilde\varphi\tilde\nabla_{X}Y-\tilde\varphi\tilde\nabla_{Y}X-\tilde\varphi\tilde
h\tilde\nabla_{X}Y+\tilde\varphi\tilde h\tilde\nabla_{Y}X\\
&=-(\tilde\nabla_{X}\tilde\varphi)Y+(\tilde\nabla_{Y}\tilde\varphi)X+\tilde\nabla_{X}\tilde\varphi\tilde
h Y-\tilde\varphi\tilde\nabla_{X}\tilde h
Y+\tilde\varphi\tilde\nabla_{X}\tilde h
Y-\tilde\nabla_{Y}\tilde\varphi\tilde h X
+\tilde\varphi\tilde\nabla_{Y}\tilde h
X\\
&\quad-\tilde\varphi\tilde\nabla_{Y}\tilde h X-\tilde\varphi\tilde
h\tilde\nabla_{X}Y+\tilde\varphi\tilde h\tilde\nabla_{Y}X\\
&=-(\tilde\nabla_{X}\tilde\varphi)Y+(\tilde\nabla_{Y}\tilde\varphi)X+(\tilde\nabla_{X}\tilde\varphi)\tilde
h Y+\tilde\varphi((\tilde\nabla_{X}\tilde
h)Y)-(\tilde\nabla_{Y}\tilde\varphi)\tilde h
X-\tilde\varphi((\tilde\nabla_{Y}\tilde h)X).
\end{align*}
Therefore, replacing \eqref{formula1} and \eqref{formula2} in
\eqref{formulapreliminare} and using the second formula in
\eqref{formule}, we obtain
\begin{align*}
\tilde R_{X Y}\xi&=\tilde g(X-\tilde h X,Y)\xi-\eta(Y)(X-\tilde h
X)-\tilde g(Y-\tilde h Y,X)\xi+\eta(X)(Y-\tilde h Y)\\
&\quad-\tilde g(X-\tilde h X,\tilde h Y)\xi+\eta(\tilde h
Y)(X-\tilde h X)-\eta(Y)(\tilde\varphi^2\tilde h
X-\tilde\varphi^2\tilde h^2 X)-2\eta(X)\tilde\varphi^2\tilde h
Y\\
&\quad+\tilde g(Y-\tilde h Y,\tilde h X)\xi-\eta(\tilde h
X)(Y-\tilde h Y)+\eta(X)(\tilde\varphi^2\tilde h
Y-\tilde\varphi^2\tilde h^2 Y)+2\eta(Y)\tilde\varphi^2\tilde h
X\\
&=\tilde g(X,Y)\xi-\tilde g(\tilde h X,Y)\xi-\eta(Y)X+\eta(Y)\tilde
h X-\tilde g(Y,X)\xi+\tilde g(\tilde h
Y,X)\xi+\eta(X)Y\\
&\quad-\eta(X)\tilde h Y-\tilde g(X,\tilde h Y)\xi+\tilde g(\tilde h
X,\tilde h Y)\xi-\eta(Y)\tilde\varphi^2\tilde h
X+\eta(Y)\tilde\varphi^2\tilde h^2 X-2\eta(X)\tilde\varphi^2\tilde h
Y\\
&\quad+\tilde g(Y,\tilde h X)\xi-\tilde g(\tilde h Y,\tilde h
X)\xi+\eta(X)\tilde\varphi^2\tilde
h Y-\eta(X)\tilde\varphi^2\tilde h^2 Y+2\eta(Y)\tilde\varphi^2\tilde h X\\
&=-\eta(Y)X+\eta(Y)\tilde h X+\eta(X)Y-\eta(X)\tilde h
Y-2\eta(X)\tilde h Y-\eta(Y)\tilde h X+\eta(Y)\tilde h^2 X\\
&\quad+2\eta(Y)\tilde h X+\eta(X)\tilde h Y-\eta(X)\tilde h^2 Y\\
&=-\eta(Y)X+\eta(X)Y+\left(1-\kappa-\left(1-\frac{\mu}{2}\right)^2\right)\eta(Y)\varphi^2
X-\left(1-\kappa-\left(1-\frac{\mu}{2}\right)^2\right)\eta(X)\varphi^2
Y\\
&\quad-2\eta(X)\tilde h Y+2\eta(Y)\tilde h X\\
&=\left(\kappa-2+\left(1-\frac{\mu}{2}\right)^2\right)\left(\eta(Y)X-\eta(X)Y\right)+2\left(\eta(Y)\tilde
h X-\eta(X)\tilde h Y\right).
\end{align*}
\end{proof}

Theorem \ref{principale2} justifies the following definition. A
paracontact metric manifold $(M,\tilde\varphi,\xi,\eta,\tilde g)$ is
said to be a \emph{paracontact metric
$(\tilde\kappa,\tilde\mu)$-manifold} if the curvature tensor field
of the Levi Civita connection satisfies
\begin{equation}
\label{km} \tilde R_{XY}\xi=\tilde\kappa (\eta(Y) X -\eta (X) Y)
+\tilde\mu (\eta(Y) \tilde hX -\eta (X)\tilde hY),
\end{equation}
where $\tilde\kappa$, $\tilde\mu$ are real constants. Using
\eqref{km} and the formula (cf. \cite{zamkovoy})
\begin{equation}
\label{RZ} \tilde R_{\xi X}\xi+\tilde\varphi \tilde R_{\xi
\tilde\varphi X}\xi=2(\tilde\varphi^2 X-\tilde h^2X),
\end{equation}
one can easily prove that
\begin{equation}\label{para1}
\tilde h^2=(1+\tilde\kappa)\tilde\varphi^2.
\end{equation}
In particular for $\tilde\kappa=-1$ we get $\tilde h^2=0$ and now
the analogy with contact metric $(\kappa,\mu)$-manifolds breaks down
because, since the metric $\tilde g$ is not positive definite, we
can not conclude that $\tilde h=0$ and the manifold is
para-Sasakian. Natural questions may be whether there exist explicit
examples of paracontact metric manifolds such that $\tilde h^2=0$
but $\tilde h\neq 0$ and whether the
$(\tilde\kappa,\tilde\mu)$-nullity condition \eqref{km} could force
the operator $\tilde h$ to vanish identically even if the metric
$\tilde g$ is not positive definite. It should be also remarked that
though paracontact metric manifolds with $\tilde h^2=0$ have made
their appearance in several contexts (see for instance Theorem 3.12
of \cite{zamkovoy}), at the knowledge of the authors not even one
explicit example of them has been given. Now we provide an example
which  solves the questions stated before.

\begin{example}
\emph{Let $\mathfrak g$ be the $5$-dimensional Lie algebra with
basis ${X_1,X_2,Y_1,Y_2,\xi}$ and non vanishing  Lie brackets
defined by}
\begin{gather*}
[X_1,X_2]=2X_2,\ [X_1,Y_1]=2\xi,\ [X_2,Y_1]=-2Y_2,\
[X_2,Y_2]=2(Y_1+\xi),\\
[\xi,X_1]=-2Y_1,\ \ [\xi,X_2]=-2Y_2.
\end{gather*}
\emph{Let $G$ be a Lie group whose Lie algebra is $\mathfrak g$. On
$G$ we define a left-invariant paracontact metric structure
$(\tilde\varphi,\xi,\eta,\tilde g)$ by setting $\tilde\varphi\xi=0$
and  $\tilde\varphi X_i=X_i$, $\tilde\varphi Y_i=-Y_i$,
$\eta(X_i)=\eta(Y_i)=0$, $\eta(\xi)=1$, and $\tilde
g(X_i,X_j)=\tilde g(Y_i,Y_j)=0$, $\tilde g(X_i,Y_i)=1$, $\tilde
g(X_1,Y_2)=\tilde g(X_2,Y_1)=0$, for all $i,j\in\{1,2\}$. Then a
direct computation shows that $\tilde h^2$ vanishes identically, but
$\tilde h\neq 0$ since, for example, $\tilde h X_1=-Y_1$. Moreover,
one can verify that $(G,\tilde\varphi,\xi,\eta,\tilde g)$ is a
paracontact metric $(\tilde\kappa,\tilde\mu)$-manifold, with
$\tilde\kappa=-1$ and $\tilde\mu=2$.}
\end{example}

\section{The canonical sequence of contact and paracontact metric structures\\ associated with a contact metric
$(\kappa,\mu)$-space}\label{sezioneprincipale}

In this section we will show that in fact the procedure of Theorem
\ref{principale1}, used for defining the canonical paracontact
metric structure $(\tilde\varphi,\xi,\eta,\tilde g)$ via the Lie
derivative of $\varphi$, can be iterated. Indeed, Lemma
\ref{lemmaluigia} suggests that the Lie derivative of
$\tilde\varphi$ in the direction $\xi$ could define a compatible
almost contact or paracontact structure on $(M,\eta)$ provided that
the coefficient $1-\kappa-\left(1-\frac{\mu}{2}\right)^2$, which
directly brings up the invariant $I_M$, is positive or negative,
respectively. Furthermore, we show that this algorithm can be
applied also to the new contact and paracontact structures, so that
one can attach to $M$ a canonical sequence of contact and
paracontact metric structures, which strictly depends on the
invariant $I_M$ and hence on the class of $M$ according to the
classification recalled in $\S$ \ref{primasezione}. We start by
proving the following fundamental result.

\begin{theorem}\label{principale3}
Let $(M,\varphi,\xi,\eta,g)$ be a contact metric
$(\kappa,\mu)$-manifold and let $(\tilde\varphi,\xi,\eta,\tilde
g)$ be the canonical paracontact metric structure of $M$. Then
\begin{enumerate}
  \item[(i)] if $|I_M|<1$, the paracontact metric structure
$(\tilde\varphi,\xi,\eta,\tilde g)$ induces on $(M,\eta)$ a
canonical compatible contact metric $(\kappa_1,\mu_1)$-structure
$(\varphi_1,\xi,\eta,g_1)$, where
\begin{equation}\label{costanti1}
\kappa_1=\kappa+\left(1-\frac{\mu}{2}\right)^2, \  \ \mu_1=2;
\end{equation}
  \item[(ii)] if $|I_M|>1$, the paracontact metric structure
$(\tilde\varphi,\xi,\eta,\tilde g)$ induces  on $(M,\eta)$
 a canonical compatible paracontact metric $(\tilde\kappa_1,\tilde\mu_1)$-structure
$(\tilde\varphi_1,\xi,\eta,\tilde g_1)$, where
\begin{equation}\label{costanti2}
\tilde\kappa_1=\kappa-2+\left(1-\frac{\mu}{2}\right)^2, \  \
\tilde\mu_1=2.
\end{equation}
\end{enumerate}
\end{theorem}
\begin{proof}
(i) \ Let us assume that $|I_M|<1$. Notice that by Lemma
\ref{lemmaluigia}  $\tilde h^2$ is proportional to $\varphi^2$ and
the  constant of proportionality
$-\left(2-\mu\right)^2+4\left(1-\kappa\right)$ is positive since we
are assuming that $|I_M|<1$.  Then we set
\begin{align}\label{definzionephi}
\varphi_1:&=\frac{1}{\sqrt{1-\kappa-\left(1-\frac{\mu}{2}\right)^2}}
\tilde
h\\
&=\frac{1}{2\sqrt{(1-\kappa)\left(1-\kappa-\left(1-\frac{\mu}{2}\right)^2\right)}}((2-\mu)\varphi
\circ h+2(1-\kappa)\varphi).\nonumber
\end{align}
Due to \eqref{formule} we have that
$\varphi_1^2=\varphi^2=-I+\eta\otimes\xi$, hence
$(\varphi_1,\xi,\eta)$ is an almost contact structure on $M$. We
look forward a compatible Riemannian metric $g_1$ such that
$d\eta=g_1(\cdot,\varphi_1\cdot)$. Thus we set
\begin{equation}\label{metrica}
g_1(X,Y):=-d\eta(X,\varphi_1 Y)+\eta(X)\eta(Y).
\end{equation}
We first need to prove that $g_1$ is a Riemannian metric. For any
$X,Y\in\Gamma(TM)$, using the symmetry of the operator $\tilde h$
 with respect to $\tilde g$, we have
\begin{align*}
g_1(Y,X)&=-\frac{1}{\sqrt{1-\kappa-\left(1-\frac{\mu}{2}\right)^2}}
d\eta(Y,\tilde
h X)+\eta(Y)\eta(X)\\
&=-\frac{1}{\sqrt{1-\kappa-\left(1-\frac{\mu}{2}\right)^2}}\tilde
g(Y,\tilde\varphi
\tilde h X)+\eta(Y)\eta(X)\\
&=-\frac{1}{\sqrt{1-\kappa-\left(1-\frac{\mu}{2}\right)^2}}\tilde g(X,\tilde\varphi\tilde h Y)+\eta(X)\eta(Y)\\
&=-d\eta(X,\varphi_1 Y)+\eta(X)\eta(Y)\\
&=g_1(X,Y),
\end{align*}
so that $g_1$ is a symmetric tensor. \ Moreover, directly by
\eqref{metrica}, $d\eta(X,Y)=g_1(X,\varphi_1 Y)$ and $g_1(\varphi_1
X,\varphi_1 Y)=g_1(X,Y)-\eta(X)\eta(Y)$ for all $X,Y\in\Gamma(TM)$.
Now we look forward conditions ensuring the positive definiteness of
$g_1$. Let $X$ be a non-zero vector field on $M$ and put
$\alpha:=\frac{1}{2\sqrt{\left(1-\kappa\right)\left(1-\kappa-\left(1-\frac{\mu}{2}\right)^2\right)}}$.
Since $g(\xi,\xi)=\eta(\xi)\eta(\xi)=1>0$ we can assume that
$X\in\Gamma({\mathcal D})$. Then by \eqref{definzionephi} and
\eqref{metrica}
\begin{align}\label{confronto}
g_1(X,X)&=-\alpha(2-\mu)d\eta(X,\varphi h
X)-2\alpha(1-\kappa)d\eta(X,\varphi
X)\nonumber\\
&=\alpha(2-\mu)g(X,hX)+2\alpha(1-\kappa)g(X,X)\nonumber\\
&=\alpha(2-\mu)g(X_\lambda+X_{-\lambda},h(X_{\lambda}+X_{-\lambda}))+2\alpha(1-\kappa)g(X_\lambda+X_{-\lambda},X_\lambda+X_{-\lambda})\\
&=\alpha(2-\mu)g(X_\lambda+X_{-\lambda},\lambda X_{\lambda}-\lambda X_{-\lambda})+2\alpha(1-\kappa)g(X_\lambda+X_{-\lambda},X_\lambda+X_{-\lambda})\nonumber\\
&=\alpha\lambda(2\lambda-\mu+2)g(X_{\lambda},X_{\lambda})+\alpha\lambda(2\lambda+\mu-2)g(X_{-\lambda},X_{-\lambda}),\nonumber
\end{align}
where we have decomposed the vector field $X\in\Gamma({\mathcal D})$
into its components along ${\mathcal D}(\lambda)$ and ${\mathcal
D}(-\lambda)$, $\lambda=\sqrt{1-\kappa}$. Thus $g_1$ is a Riemannian
metric provided that  $2\lambda-\mu+2>0$ and $2\lambda+\mu-2>0$. In
view of \eqref{invariante1}--\eqref{invariante2}, the above
conditions are just equivalent to the positive definiteness of the
Legendre foliation ${\mathcal D}(\lambda)$ and to the negative
definiteness of ${\mathcal D}(-\lambda)$, and hence to the
requirement that $|I_M|<1$. Thus, as we are assuming that $|I_M|<1$,
we conclude that $g_1$ is a Riemannian metric. We now prove that
$(\varphi_1,\xi,\eta,g_1)$ is a contact metric
$(\kappa_1,\mu_1)$-structure, for some constants $\kappa_1$ and
$\mu_1$ to be found. For this purpose we firstly find a more
explicit expression of  the tensor field $h_1:=\frac{1}{2}{\mathcal
L}_{\xi}\varphi_1$. As before, set
$\alpha:=\frac{1}{2\sqrt{\left(1-\kappa\right)\left(1-\kappa-\left(1-\frac{\mu}{2}\right)^2\right)}}$.
Then taking  \eqref{paracanonical} and \eqref{formule} into account,
one has
\begin{align*}
h_1&=\frac{\alpha}{2}\left((2-\mu)\left(({\mathcal
L}_{\xi}\varphi)\circ h + \varphi\circ({\mathcal L}_{\xi}
h)\right)+2(1-\kappa){\mathcal L}_{\xi}\varphi\right)\\
&=\frac{\alpha}{2}\left((2-\mu)(2h^2+(2-\mu)\varphi^2\circ
h+2(1-\kappa)\varphi^2)+4(1-\kappa)h\right)\\
&=\frac{\alpha}{2}\left(-(2-\mu)^2+4(1-\kappa)\right)h\\
&=\sqrt{1-{I_M}^2}h.
\end{align*}
Thus $h_1$ is proportional to $h$ and hence it admits the
eigenvalues $\lambda_1$ and $-\lambda_1$, where
$\lambda_1:=\sqrt{(1-\kappa)(1-{I_M}^2)}=1-\kappa-(1-\frac{\mu}{2})^2$,
and the corresponding eigendistributions coincide with the
eigendistributions of the operator $h$.  Then the bi-Legendrian
connection associated with $({\mathcal D}(-\lambda_1),{\mathcal
D}(\lambda_1))$  coincides with the bi-Legendrian connection
$\nabla^{bl}$ associated with the bi-Legendrian structure
$({\mathcal D}(-\lambda),{\mathcal D}(\lambda))$ induced by $h$. We
prove that $\nabla^{bl}$ preserves the tensor fields $\varphi_1$.
Indeed for all $X,Y\in\Gamma(TM)$
\begin{align*}
(\nabla^{bl}_{X}\varphi_1)Y=\alpha(2-\mu)\left((\nabla^{bl}_{X}\varphi)h
Y+\varphi(\nabla^{bl}_{X}
h)Y\right)+2\alpha(1-\kappa)(\nabla^{bl}_{X}\varphi)Y=0
\end{align*}
since $\nabla^{bl}\varphi=0$ and $\nabla^{bl}h=0$. Moreover, as
$\nabla^{bl}\varphi_1=0$ and $\nabla^{bl}d\eta=0$, also
$\nabla^{bl}g_1=0$. Therefore, since obviously also
$\nabla^{bl}h_1=0$, $\nabla^{bl}$ satisfies all the conditions of
Theorem \ref{principale0} and we can conclude that
$(\varphi_1,\xi,\eta,g_1)$ is a contact metric
$(\kappa_1,\mu_1)$-structure. In order to find the expression of
$\kappa_1$ and $\mu_1$, we observe that, immediately,
$\kappa_1=1-\lambda_1^2=\kappa+(1-\kappa){I_{M}}^2=\kappa+\left(1-\frac{\mu}{2}\right)^2$.
Then applying \eqref{invariante1} and $\Pi_{{\mathcal
D}(\lambda)}=\Pi_{{\mathcal D}(\lambda_1)}$,  we have, for any non
zero $X\in\Gamma({\mathcal D}(\lambda))$,
$(2\sqrt{1-\kappa}-\mu+2)g(X,X)=(2\sqrt{1-\kappa_1}-\mu_1+2)g_1(X,X)$.
Using \eqref{confronto} we get
$2\sqrt{1-\kappa_1}-\mu_1+2=\sqrt{-(2-\mu)^2+4(1-\kappa)}$, so that
$\mu_1=2\sqrt{1-\kappa-\left(1-\frac{\mu}{2}\right)^2}+2-\sqrt{-(2-\mu)^2+4(1-\kappa)}=2$.\\
(ii) \ Now let us assume that $|I_M|>1$. Then we define
\begin{align}\label{definzioneparaphi}
\tilde\varphi_1:&=\frac{1}{\sqrt{\left(1-\frac{\mu}{2}\right)^2-(1-\kappa)}}\tilde
h\\
&=\frac{1}{2\sqrt{(1-\kappa)\left(\left(1-\frac{\mu}{2}\right)^2-(1-\kappa)\right)}}((2-\mu)\varphi\circ
h+2(1-\kappa)\varphi).\nonumber
\end{align}
Using \eqref{formule} and the assumption $|I_M|>1$, one easily
proves that $\tilde\varphi_1^2=I-\eta\otimes\xi$, so that in order
to conclude that $(\tilde\varphi_1,\xi,\eta)$ defines an almost
paracontact structure we need only to prove that the
eigendistributions corresponding to the eigenvalues $1$ and $-1$ of
$\tilde\varphi_1|_{\mathcal D}$ have equal dimension $n$. Notice
that thought $\tilde h$ is a symmetric operator (with respect to
$\tilde g$) it could be not necessarily diagonalizable, since
$\tilde g$ is not positive definite. Nevertheless we now show that
this is the case. Let $\{X_1,\ldots,X_n,Y_1,\ldots,Y_n,\xi\}$ be a
local orthonormal $\varphi$-basis of eigenvectors of $h$, namely
$X_i=-\varphi Y_i$, $Y_i=\varphi X_i$, $h X_i=\lambda X_i$, $h
Y_i=-\lambda Y_i$, $i\in\left\{1,\ldots,n\right\}$. Then, by
\eqref{formule}, for each $i\in\left\{1,\ldots,n\right\}$,
\begin{align*}
\tilde h X_i&=\frac{1}{2\sqrt{1-\kappa}}\left((2-\mu)\varphi h
X_i+2(1-\kappa)\varphi X_i\right)\\
&=\frac{1}{2\sqrt{1-\kappa}}\left((2-\mu)\lambda
Y_i+2(1-\kappa)Y_i\right)\\
&=\left(1-\frac{\mu}{2}+\sqrt{1-\kappa}\right)Y_i
\end{align*}
and, analogously, one finds  $\tilde h Y_i =
\left(1-\frac{\mu}{2}-\sqrt{1-\kappa}\right)X_i$. Hence $\tilde h$
is represented, with respect to the basis
$\{X_1,\ldots,X_n,Y_1,\ldots,Y_n,\xi\}$, by the matrix
\begin{equation*}
\left(
  \begin{array}{ccc}
    0_n & \left(1-\frac{\mu}{2}-\sqrt{1-\kappa}\right)I_n & \textbf{0}_{n 1} \\
    \left(1-\frac{\mu}{2}+\sqrt{1-\kappa}\right)I_n & \textbf{0}_n & \textbf{0}_{n 1} \\
    \textbf{0}_{1 n} & \textbf{0}_{1 n} & 0 \\
  \end{array}
\right),
\end{equation*}
where $\textbf{0}_n, \textbf{0}_{n 1}, \textbf{0}_{1 n}$ denote,
respectively, the $n\times n$, $n\times 1$ and $1\times n$ matrices
whose entries are all $0$, and $I_n$ the identity matrix of order
$n$. Therefore the characteristic polynomial is given by
\begin{equation*}
p=-\lambda\left(\lambda^2-\left(1-\frac{\mu}{2}+\sqrt{1-\kappa}\right)\left(1-\frac{\mu}{2}-\sqrt{1-\kappa}\right)\right)^n=-\lambda\left(\lambda^2-\left(\left(1-\frac{\mu}{2}\right)^2-(1-\kappa)\right)\right)^n.
\end{equation*}
Because of the assumption $|I_M|>1$, the number
$\left(1-\frac{\mu}{2}\right)^2-(1-\kappa)$ is  positive, so that
the operator $\tilde h$ admits, apart from the eigenvalue $0$
corresponding to the eigenvector $\xi$, also the eigenvalues
$\tilde\lambda$ and $-\tilde\lambda$, where
$\tilde\lambda:=\sqrt{\left(1-\frac{\mu}{2}\right)^2-(1-\kappa)}$.
An easy computation shows that the corresponding eigendistributions
are, respectively,
\begin{gather}\label{autospazio1}
{\mathcal
D}(\tilde\lambda)=\textrm{span}\left\{\sqrt\frac{I_M-1}{I_M+1}X_1+Y_1,\ldots,\sqrt\frac{I_M-1}{I_M+1}X_n+Y_n\right\},\\
{\mathcal
D}(-\tilde\lambda)=\textrm{span}\left\{-\sqrt\frac{I_M-1}{I_M+1}X_1+Y_1,\ldots,-\sqrt\frac{I_M-1}{I_M+1}X_n+Y_n\right\}.
\label{autospazio2}
\end{gather}
Therefore each eigendistribution ${\mathcal D}(\tilde\lambda)$ and
${\mathcal D}(-\tilde\lambda)$ has dimension $n$ and finally this
implies in turn that the eigendistributions of the operator
$\tilde\varphi_1$ restricted to $\mathcal D$ are $n$-dimensional.
Thus $(\tilde\varphi_1,\xi,\eta)$ is an almost paracontact
structure. Next we define a compatible semi-Riemannian metric by
putting, for any $X,Y\in\Gamma(TM)$,
\begin{equation}\label{parametrica}
\tilde g_1(X,Y):=d\eta(X,\tilde\varphi_1 Y)+\eta(X)\eta(Y).
\end{equation}
That $\tilde g_1$ is  symmetric can be easily proved. Moreover,
directly from \eqref{parametrica} one can show that, for all
$X,Y\in\Gamma(TM)$, $\tilde g_1(\tilde\varphi_1 X,\tilde\varphi_1
Y)=-g_1(X,Y)+\eta(X)\eta(Y)$ and $d\eta(X,Y)=\tilde
g_1(X,\tilde\varphi_1 Y)$. Therefore
$(\tilde\varphi_1,\xi,\eta,\tilde g_1)$ is a paracontact metric
structure on $M$. We notice also that, arguing as in the previous
case, one can find that
\begin{equation*}
\tilde
h_1=\frac{1}{4\sqrt{(1-\kappa)\left(\left(1-\frac{\mu}{2}\right)^2-4(1-\kappa)\right)}}\left((2-\mu){\mathcal
L}_{\xi}(\varphi\circ h)+2(1-\kappa){\mathcal
L}_{\xi}\varphi\right)=-\sqrt{{I_M}^2-1} h.
\end{equation*}
It remains to show that $(M,\tilde\varphi_1,\xi,\eta,\tilde g_1)$
verifies a $(\tilde\kappa_1,\tilde\mu_1)$-nullity condition, for
some constants $\tilde\kappa_1$ and $\tilde\mu_1$. For this purpose
we find the relationship between the Levi Civita connections
$\tilde\nabla$ and $\tilde\nabla^1$ of $\tilde g$ and $\tilde g_1$,
respectively. Notice that, by \eqref{parametrica},
\begin{align}\label{parametricabis}
\tilde
g_1(X,Y)=\frac{1}{\sqrt{\left(1-\frac{\mu}{2}\right)^2-\left(1-\kappa\right)}}d\eta(X,\tilde
h Y)+\eta(X)\eta(Y)=\beta\tilde g(X,\tilde\varphi\tilde h
Y)+\eta(X)\eta(Y),
\end{align}
where we have put
$\beta:=\frac{1}{\sqrt{\left(1-\frac{\mu}{2}\right)^2-\left(1-\kappa\right)}}$.
Then, arguing as in Proposition \ref{levicivita1}, we have, for all
$X,Y,Z\in\Gamma(TM)$,
\begin{align*}
2\tilde g_1(\tilde\nabla^1_{X}Y,Z)&=\beta\left(2\tilde
g(\tilde\varphi\tilde h \tilde\nabla_{X}Y,Z)+\tilde
g(Y,(\tilde\nabla_{X}\tilde\varphi\tilde h)Z)+\tilde
g(X,(\tilde\nabla_{Y}\tilde\varphi\tilde h)Z)-\tilde
g(X,(\tilde\nabla_{Z}\tilde\varphi\tilde h)Y)\right)\\
&\quad+2\left(d\eta(X,Z)\eta(Y)+d\eta(Y,Z)\eta(X)-d\eta(X,Y)\eta(Z)+X(\eta(Y))\eta(Z)\right).
\end{align*}
Using \eqref{formula1}, \eqref{formula2} and the identity
$(\tilde\nabla_{X}\tilde\varphi\tilde
h)Y=(\tilde\nabla_{X}\tilde\varphi)\tilde
hY+\tilde\varphi((\tilde\nabla_{X}\tilde h)Y)$, the previous
relation becomes
\begin{align}\label{relazione1}
2\tilde g_1(\tilde\nabla^1_{X}Y,Z)&=\beta\bigl(2\tilde
g(\tilde\varphi\tilde h\tilde\nabla_{X}Y,Z)-\eta(Y)\tilde g(X,\tilde
hZ)+\eta(Y)\tilde g(\tilde hX,\tilde hZ)-2\eta(X)\tilde
g(Y,\tilde\varphi^2\tilde h
Z)\nonumber\\
&\quad-\eta(Z)\tilde g(Y,\tilde\varphi^2\tilde h X)+\eta(Z)\tilde
g(Y,\tilde\varphi^2\tilde h^2 X)-\eta(X)\tilde g(Y,\tilde h
Z)+\eta(X)\tilde g(\tilde hY,\tilde hZ)\nonumber\\
&\quad-2\eta(Y)\tilde g(X,\tilde\varphi^2\tilde hZ)-\eta(Z)\tilde
g(X,\tilde\varphi^2\tilde h Y)+\eta(Z)\tilde
g(X,\tilde\varphi^2\tilde h^2 Y)+\eta(X)\tilde g(Z,\tilde hY)\\
&\quad-\eta(X)\tilde g(\tilde hZ,\tilde hY)+2\eta(Z)\tilde
g(X,\tilde\varphi^2\tilde hY)+\eta(Y)\tilde
g(X,\tilde\varphi^2\tilde h Z)-\eta(Y)\tilde
g(X,\tilde\varphi^2\tilde h^2 Z)\bigr)\nonumber\\
&\quad+2\left(d\eta(X,Z)\eta(Y)+d\eta(Y,Z)\eta(X)-d\eta(X,Y)\eta(Z)+X(\eta(Y))\eta(Z)\right).\nonumber
\end{align}
Notice that, by \eqref{valori} and \eqref{para1}, $\tilde
h^2=(1+\tilde\kappa)\tilde\varphi^2=\left(\kappa-1+\left(1-\frac{\mu}{2}\right)^2\right)\tilde\varphi^2=\frac{1}{\beta^2}\tilde\varphi^2$.
Substituting this relation in \eqref{relazione1} and taking the
symmetry of the operator $\tilde h$ with respect to the
semi-Riemannian metric $\tilde g$ into account, we get
\begin{align*}
2\tilde g_1(\tilde\nabla^1_{X}Y,Z)&=\beta\bigl(2\tilde
g(\tilde\varphi\tilde h\tilde\nabla_{X}Y,Z)-2\eta(X)\tilde g(\tilde
hY,Z)+\frac{2}{\beta^2}\tilde
g(X,Y)\eta(Z)-\frac{2}{\beta^2}\eta(X)\eta(Y)\eta(Z)\\
&\quad-2\eta(Y)\tilde g(\tilde
hX,Z)\bigl)+2\bigl(d\eta(X,Z)\eta(Y)+d\eta(Y,Z)\eta(X)-d\eta(X,Y)\eta(Z)\\
&\quad+X(\eta(Y))\eta(Z)\bigr),
\end{align*}
that is, by definition of $\tilde g_1$,
\begin{align}\label{relazione2}\nonumber
2\bigl(c\beta\tilde g(\tilde\nabla^{1}_{X}Y,\tilde\varphi\tilde h
Z)+\eta(\tilde\nabla^{1}_{X}Y)\tilde
g(\xi,Z)\bigr)&=\beta\bigl(2\tilde g(\tilde\varphi\tilde h
\tilde\nabla_{X}Y,Z)-2\eta(X)\tilde g(\tilde
hY,Z)+\frac{2}{\beta^2}\tilde g(X,Y)\tilde g(\xi,Z)\\
&\quad-\frac{2}{\beta^2}\eta(X)\eta(Y)\tilde g(\xi,Z)-2\eta(Y)\tilde
g(\tilde hX,Z)\bigr)\nonumber\\
&\quad+2\bigl(-\eta(Y)\tilde g(\tilde\varphi X,Z)-\eta(X)\tilde
g(\tilde\varphi Y,Z)-\tilde g(X,\tilde\varphi Y)\tilde g(\xi,Z)\\
&\quad+X(\eta(Y))\tilde g(\xi,Z)\bigr).\nonumber
\end{align}
Therefore, since $Z$ was chosen arbitrarily, we get
\begin{align}\label{relazione2bis}
\nonumber \beta\tilde\varphi\tilde
h\tilde\nabla^{1}_{X}Y+\eta(\tilde\nabla^{1}_{X}Y)\xi&=\beta\tilde\varphi\tilde
h\tilde\nabla_{X}Y-\beta\eta(X)\tilde hY+\frac{1}{\beta}\tilde
g(X,Y)\xi-\frac{1}{\beta}\eta(X)\eta(Y)\xi-\beta\eta(Y)\tilde h X\\
&\quad-\eta(Y)\tilde\varphi X-\eta(X)\tilde\varphi Y-\tilde
g(X,\tilde\varphi Y)\xi+X(\eta(Y))\xi.
\end{align}
Note that, since $\tilde\varphi_1=\beta\tilde h$, $\tilde
h_1=-\frac{1}{\beta}\tilde\varphi$ and $\tilde
h^2=\frac{1}{\beta^2}\tilde\varphi^2$,
\begin{align}\label{relazione3}
\eta(\tilde\nabla^{1}_{X}Y)&=\tilde g_1(\tilde\nabla^{1}_{X}Y,\xi)\nonumber\\
&=X(\tilde g_1(Y,\xi))-\tilde g_1(Y,\tilde\nabla^{1}_{X}\xi)\nonumber\\
&=X(\eta(Y))-\tilde g_1(Y,-\tilde\varphi_1 X+\tilde\varphi_1\tilde
h_1 X)\nonumber\\
&=X(\eta(Y))+d\eta(Y,X)-\tilde g_1(Y,\tilde\varphi\tilde h X)\\
&=X(\eta(Y))-\tilde g(X,\tilde\varphi Y)-\beta\tilde
g(Y,\tilde\varphi\tilde h\tilde\varphi\tilde h X)\nonumber\\
&=X(\eta(Y))-\tilde g(X,\tilde\varphi Y)+\frac{1}{\beta}\tilde
g(X,Y)-\frac{1}{\beta}\eta(X)\eta(Y).\nonumber
\end{align}
Consequently, \eqref{relazione2bis} becomes
\begin{equation*}
\tilde h \tilde\nabla^{1}_{X}Y=\tilde
h\tilde\nabla_{X}Y-\eta(X)\tilde\varphi\tilde h
Y-\eta(Y)\tilde\varphi\tilde h
X-\frac{1}{\beta}\eta(Y)\tilde\varphi^2
X-\frac{1}{\beta}\eta(X)\tilde\varphi^2 Y.
\end{equation*}
Applying $\tilde h$ we obtain
\begin{equation}\label{relazione4}
\tilde\nabla^{1}_{X}Y-\eta(\tilde\nabla^{1}_{X}Y)\xi=\tilde\nabla_{X}Y-\eta(\tilde\nabla_{X}Y)\xi+\eta(X)\tilde\varphi
Y +\eta(Y)\tilde\varphi X-\beta\eta(Y)\tilde hX-\beta\eta(X)\tilde h
Y.
\end{equation}
Now, a straightforward computation as in \eqref{relazione3} shows
that
\begin{equation}\label{relazione5}
\eta(\tilde\nabla_{X}Y)=X(\eta(Y))-\tilde g(X,\tilde\varphi
Y)-\tilde g(X,\tilde\varphi\tilde h Y).
\end{equation}
Therefore, by replacing \eqref{relazione3} and  \eqref{relazione5}
in \eqref{relazione4} and recalling that
$\beta=\frac{1}{\sqrt{\left(1-\frac{\mu}{2}\right)^2-\left(1-\kappa\right)}}$,
we finally find
\begin{align}\label{relazionefinale}
\tilde\nabla^{1}_{X}Y&=\tilde\nabla_{X}Y+\eta(X)\left(\tilde\varphi
Y-\frac{\tilde h
Y}{\sqrt{\left(1-\frac{\mu}{2}\right)^2-\left(1-\kappa\right)}}\right)+\eta(Y)\left(\tilde\varphi
X-\frac{\tilde h
X}{\sqrt{\left(1-\frac{\mu}{2}\right)^2-\left(1-\kappa\right)}}\right)\nonumber\\
&\quad+\left(\sqrt{\left(1-\frac{\mu}{2}\right)^2-\left(1-\kappa\right)}\bigl(\tilde
g(X,Y)-\eta(X)\eta(Y)\bigr)+\tilde g (X,\tilde\varphi\tilde h
Y)\right)\xi.
\end{align}
The explicit expression \eqref{relazionefinale} of the Levi Civita
connection of $\tilde g_1$ in terms of the Levi Civita connection of
$\tilde g$ allows us to prove that
$(M,\tilde\varphi_1,\xi,\eta,\tilde g_1)$ is a paracontact metric
$(\tilde\kappa_1,\tilde\mu_1)$-manifold, for some $\tilde\kappa_1,
\tilde\mu_1 \in \mathbb{R}$. Indeed, from \eqref{relazionefinale},
after some long but straightforward computations, we obtain
\begin{align}\label{relazione6}
(\tilde\nabla^{1}_{X}\tilde\varphi_1)Y&=\left(-\frac{1}{\sqrt{\left(1-\frac{\mu}{2}\right)^2-\left(1-\kappa\right)}}\tilde
g(X,\tilde\varphi\tilde h Y)-\eta(X)\eta(Y)+\tilde g(X,\tilde h
Y)\right)\xi\nonumber\\
&\quad
+\eta(Y)\left(X+\sqrt{\left(1-\frac{\mu}{2}\right)^2-\left(1-\kappa\right)}\tilde\varphi
X\right)\nonumber\\
&=-\tilde g_1(X-\tilde h_1 X,Y)\xi+\eta(Y)(X-\tilde h_1 X),
\end{align}
and
\begin{align}\label{relazione7}
(\tilde\nabla^{1}_{X}\tilde
h_1)Y&=\sqrt{\left(1-\frac{\mu}{2}\right)^2-\left(1-\kappa\right)}\eta(Y)\tilde
h X - 2\eta(X)\tilde\varphi\tilde h Y - \eta(Y)\tilde\varphi\tilde h
X\nonumber\\
&+\sqrt{\left(1-\frac{\mu}{2}\right)^2-\left(1-\kappa\right)}\left(\tilde
g(X,Y)-\eta(X)\eta(Y)-\sqrt{\left(1-\frac{\mu}{2}\right)^2-\left(1-\kappa\right)}\tilde
g(X,\tilde\varphi Y)\right)\xi\nonumber\\
&=-\eta(Y)(\tilde\varphi_1\tilde h_1 X-\tilde\varphi_1\tilde h_1^2
X)-2\eta(X)\tilde\varphi_1\tilde h_1 Y-\tilde
g_1(X,\tilde\varphi_1\tilde h_1 Y+\tilde\varphi_1\tilde h_1^2 Y)\xi.
\end{align}
Then by \eqref{formulapreliminare}, \eqref{relazione6} and
\eqref{relazione7}, and since $\tilde
h_1^2=\left(\left(1-\frac{\mu}{2}\right)^2-(1-\kappa)\right)\tilde\varphi^2$,
we get
\begin{align*}
\tilde R^{1}_{X
Y}\xi&=-(\tilde\nabla^{1}_{X}\tilde\varphi_1)Y+(\tilde\nabla^{1}_{Y}\tilde\varphi_1)X+(\tilde\nabla^{1}_{X}\tilde\varphi_1)\tilde
h Y+\tilde\varphi_1((\tilde\nabla^{1}_{X}\tilde
h)Y)-(\tilde\nabla^{1}_{Y}\tilde\varphi_1)\tilde h_1
X-\tilde\varphi_1((\tilde\nabla^{1}_{Y}\tilde h_1)X)\\
&=-\eta(Y)X+\eta(X)Y+\eta(Y)\tilde h_1^2 X-\eta(X)\tilde h_1^2
Y-2\eta(X)\tilde h_1 Y+2\eta(Y)\tilde h_1 X\\
&=-\eta(Y)X+\eta(X)Y+
\left(\left(1-\frac{\mu}{2}\right)^2-(1-\kappa)\right)\left(\eta(Y)\tilde\varphi^2
X-\eta(X)\tilde\varphi^2 Y\right)\\
&\quad-2\eta(X)\tilde h_1
Y+2\eta(Y)\tilde h_1 X\\
&=\left(\kappa-2+\left(1-\frac{\mu}{2}\right)^2\right)\bigl(\eta(Y)X-\eta(X)Y\bigr)+2\bigl(\eta(Y)\tilde
h_1 X-\eta(X)\tilde h_1 Y\bigr).
\end{align*}
Thus $(\tilde\varphi_1,\xi,\eta,\tilde g_1)$ is paracontact metric
$(\tilde\kappa_1,\tilde\mu_1)$-structure with
$\tilde\kappa_1=\kappa-2+\left(1-\frac{\mu}{2}\right)^2$ and
$\tilde\mu_1=2$.
\end{proof}

We  recall the definition of Tanaka-Webster parallel space, recently
introduced by Boeckx and Cho (\cite{Boeckx-08}). A contact metric
manifold is a \emph{Tanaka-Webster parallel space} if its
generalized Tanaka-Webster torsion  $\hat{T}$ and  curvature
$\hat{R}$ satisfy $\hat{\nabla} \hat{T}=0$ and $\hat{\nabla}
\hat{R}=0$, that is the Tanaka-Webster connection $\hat{\nabla}$ is
invariant by parallelism (in the sense of \cite{kobayashi1}). Boeckx
and Cho have proven that a contact metric manifold $M$ is a
Tanaka-Webster parallel space if and only if $M$ is a Sasakian
locally $\varphi$-symmetric space or a non-Sasakian
$(\kappa,2)$-space (\cite[Theorem 12]{Boeckx-08}). Thus, in
particular, we deduce that the contact metric
$(\kappa_1,\mu_1)$-structure $(\varphi_1,\xi,\eta,g_1)$ in (i) of
Theorem \ref{principale3} is in fact a Tanaka-Webster parallel
structure. Therefore we have proven the following corollary.

\begin{corollary}\label{tanakaparallel}
Every non-Sasakian contact metric $(\kappa,\mu)$-manifold
$(M,\varphi,\xi,\eta,g)$ such that $|I_M|<1$ admits a compatible
Tanaka-Webster parallel structure.
\end{corollary}

\begin{remark}\label{osservazione1}
We point out that in the proof of Theorem \ref{principale3} we have
proved that, even if the metric $\tilde g$ is not positive definite,
in the case $|I_M|>1$ the operator $\tilde h$ is diagonalizable and
admits the eigenvalue $0$ of multiplicity $1$ and the eigenvalues
$\tilde\lambda$ and $-\tilde\lambda$, where
$\tilde\lambda=\sqrt{\left(1-\frac{\mu}{2}\right)^2-(1-\kappa)}$,
both of multiplicity $n$. The explicit expressions of the
eigendistributions ${\mathcal D}(\tilde\lambda)$ and ${\mathcal
D}(-\tilde\lambda)$ in terms of a local $\varphi$-basis of
eigenvectors of $h$, is given by the relations
\eqref{autospazio1}--\eqref{autospazio2}. We now show that
${\mathcal D}(\tilde\lambda)$ and ${\mathcal D}(-\tilde\lambda)$ are
in fact Legendre foliations. Indeed, for any
$X,X'\in\Gamma({\mathcal D}(\tilde\lambda))$ we have
\begin{align*}
\tilde g(X,\tilde\varphi X')=\frac{1}{\tilde\lambda}\tilde
g(X,\tilde\varphi\tilde h X')=-\frac{1}{\tilde\lambda}\tilde
g(X,\tilde h\tilde\varphi X')=-\frac{1}{\tilde\lambda}\tilde
g(\tilde hX,\tilde\varphi X')=-\tilde g(X,\tilde\varphi X'),
\end{align*}
so that $\tilde g(X,\tilde\varphi X')=0$ and, consequently,
$d\eta(X,X')=0$. Analogously, for any $Y,Y'\in\Gamma({\mathcal
D}(-\tilde\lambda))$, $d\eta(Y,Y')=0$. This proves that ${\mathcal
D}(\tilde\lambda)$ and ${\mathcal D}(-\tilde\lambda)$ are Legendre
distributions. Now, observe that the almost bi-Legendrian structure
given by ${\mathcal D}(\tilde\lambda)$ and ${\mathcal
D}(-\tilde\lambda)$, by definition of $\tilde\varphi_1$, coincides
with the almost bi-Legendrian structure induced by the paracontact
metric structure $(\tilde\varphi_1,\xi,\eta,\tilde g_1)$ in Theorem
\ref{principale3}, which is integrable because of \eqref{relazione6}
and Theorem \ref{integrability}. Thus $[X,X']\in\Gamma({\mathcal
D}(\tilde\lambda)\oplus\mathbb{R}\xi)$ for all
$X,X'\in\Gamma({\mathcal D}(\tilde\lambda))$ and
$[Y,Y']\in\Gamma({\mathcal D}(-\tilde\lambda)\oplus\mathbb{R}\xi)$
for all $Y,Y'\in\Gamma({\mathcal D}(-\tilde\lambda))$. On the other
hand, since ${\mathcal D}(\tilde\lambda)$ and ${\mathcal
D}(-\tilde\lambda)$ are Legendre distributions, we have that
$\eta([X,X'])=X(\eta(X'))-X'(\eta(X))-2d\eta(X,X')=0$ and
$\eta([Y,Y'])=0$, so that $[X,X']\in\Gamma({\mathcal D})$ and
$[Y,Y']\in\Gamma({\mathcal D})$ for all $X,X'\in\Gamma({\mathcal
D}(\tilde\lambda))$, $Y,Y'\in\Gamma({\mathcal D}(-\tilde\lambda))$.
Hence we conclude that  ${\mathcal D}(\tilde\lambda)$ and ${\mathcal
D}(-\tilde\lambda)$ are involutive.
\end{remark}

Therefore, any contact metric $(\kappa,\mu)$-manifold
$(M,\varphi,\xi,\eta,g)$ with $|I_M|>1$ admits a supplementary
bi-Legendrian structure, given by the eigendistributions of the
operator $\tilde h$ of the canonical paracontact metric structure
$(\tilde\varphi,\xi,\eta,\tilde g)$ induced by
$(\varphi,\xi,\eta,g)$. But the surprising fact is that such
bi-Legendrian structure $({\mathcal D}(\tilde\lambda),{\mathcal
D}(-\tilde\lambda))$ comes from a (new) contact metric
$(\kappa',\mu')$-structure, as we now prove.

\begin{theorem}\label{legendre3}
Let $(M,\varphi,\xi,\eta,g)$ be a contact metric
$(\kappa,\mu)$-manifold such that $|I_M|>1$ and let
$(\tilde\varphi,\xi,\eta,\tilde g)$ be the canonical paracontact
metric structure induced on $M$. Then the operator $\tilde
h:=\frac{1}{2}{\mathcal L}_{\xi}\tilde\varphi$ is diagonalizable and
admits the eigenvalues $0$ of multiplicity $1$ and
$\pm\tilde\lambda$ of multiplicity $n$, where
$\tilde\lambda:=\sqrt{\left(1-\frac{\mu}{2}\right)^2-(1-\kappa)}$.
Moreover, denoting by ${\mathcal D}(\tilde\lambda)$ and ${\mathcal
D}(-\tilde\lambda)$, the eigendistributions corresponding to
$\tilde\lambda$ and $-\tilde\lambda$, respectively, there exists a
family of compatible contact metric
$(\kappa'_{a,b},\mu'_{a,b})$-structures
$(\varphi'_{a,b},\xi,\eta,g'_{a,b})$ whose associated bi-Legendrian
structure coincides with $({\mathcal D}(\tilde\lambda),{\mathcal
D}(-\tilde\lambda))$, where
\begin{equation}\label{costanti3}
\kappa'_{a,b}=1-\frac{(a-b)^2}{16}, \ \ \
\mu'_{a,b}=2-\frac{a+b}{2},
\end{equation}
$a$ and $b$ being any two positive real numbers such that
\begin{equation}\label{costante}
ab=\frac{1}{4}\left(\left(1-\frac{\mu}{2}\right)^2-(1-\kappa)\right).
\end{equation}
Furthermore,  the Boeckx invariant of
$(M,\varphi'_{a,b},\xi,\eta,g'_{a,b})$ has absolute value strictly
greater than $1$, so that $(\varphi'_{a,b},\xi,\eta,g'_{a,b})$
belongs to the same class as $(\varphi,\xi,\eta,g)$, according to
the classification in $\S$ \emph{\ref{primasezione}}.
\end{theorem}
\begin{proof}
The first part of the theorem has been already proven in Theorem
\ref{principale3} and Remark \ref{osservazione1}. The remaining part
of the proof consists in  showing that the bi-Legendrian structure
$({\mathcal D}(-\tilde\lambda),{\mathcal D}(\tilde\lambda))$
verifies the hypotheses of Theorem \ref{legendre1}. First we find
the expression of the invariants $\Pi_{{\mathcal D}(\tilde\lambda)}$
and $\Pi_{{\mathcal D}(-\tilde\lambda)}$. For any
$X,X'\in\Gamma({\mathcal D}(\tilde\lambda))$ we have
\begin{equation*}
\Pi_{{\mathcal D}(\tilde\lambda)}(X,X')=2d\eta([\xi,X],X')=2\tilde
g_1([\xi,X],\tilde\varphi_1 X')=2\tilde g_1([\xi,X], X')=2\tilde
g_1(\tilde h_1 X,X'),
\end{equation*}
and, analogously, for any $Y,Y'\in\Gamma({\mathcal
D}(-\tilde\lambda))$,
\begin{equation*}
\Pi_{{\mathcal D}(-\tilde\lambda)}(Y,Y')=2d\eta([\xi,Y],Y')=2\tilde
g_1([\xi,Y],\tilde\varphi_1 Y')=-2\tilde g_1([\xi,Y],Y')=2\tilde
g_1(\tilde h_1 Y,Y'),
\end{equation*}
where we used the easy relations $\tilde h_1 X = [\xi,X]_{{\mathcal
D}(-\tilde\lambda)}$ and $\tilde h_1 Y = -[\xi,Y]_{{\mathcal
D}(\tilde\lambda)}$, for any $X\in\Gamma({\mathcal
D}(\tilde\lambda))$ and $Y\in\Gamma({\mathcal D}(-\tilde\lambda))$.
We prove that $\nabla^{{\prime}{bl}}\Pi_{{\mathcal
D}(\tilde\lambda)}=\nabla^{{\prime}{bl}}\Pi_{{\mathcal
D}(-\tilde\lambda)}=0$, where $\nabla^{{\prime}{bl}}$ denotes the
bi-Legendrian connection associated to the bi-Legendrian structure
$({\mathcal D}(-\tilde\lambda),{\mathcal D}(\tilde\lambda))$.
Indeed, notice that, by Theorem \ref{connection} and the
integrability of $(\tilde\varphi_1,\xi,\eta,\tilde g_1)$,
$\nabla^{{\prime}{bl}}$ coincides with the canonical paracontact
connection $\tilde\nabla^{{1}{pc}}$ of
$(M,\tilde\varphi_1,\xi,\eta,\tilde g_1)$. In particular, by
\eqref{paradefinition} and \eqref{relazione7}, we have for any
$X,Y\in\Gamma(TM)$,
\begin{align*}
(\nabla^{{\prime}{bl}}_{X}\tilde
h_1)Y&=(\tilde\nabla^{{1}{pc}}_{X}\tilde h_1)Y\\
&=(\tilde\nabla^{1}_{X}\tilde h_1)Y+\eta(X)\tilde\varphi_1\tilde h_1
Y+\tilde g_1(X-\tilde h_1 X,\tilde\varphi_1\tilde h_1
Y)\xi-\eta(Y)\tilde h_1\tilde\varphi_1
Y\\
&\quad+\eta(Y)(\tilde\varphi_1\tilde h_1 X-\tilde\varphi_1\tilde
h^2_1 X)=0,
\end{align*}
where $\tilde\nabla^{1}$ denotes the Levi Civita connection of
$(M,\tilde g_1)$. Consequently, for any $X,X'\in\Gamma({\mathcal
D}(\tilde\lambda))$ and $Z\in\Gamma(TM)$,
\begin{align*}
(\nabla'^{bl}_{Z}\Pi_{{\mathcal D}(\tilde\lambda)})(X,X')&=2Z(\tilde
g_1(\tilde h_1 X,X'))-2\tilde g_1(\tilde
h_1\nabla'^{bl}_{Z}X,X')-2\tilde g_1(\tilde h_1
X,\nabla'^{bl}_{Z}X')\\
&=2\bigl(Z(\tilde g_1(\tilde h_1 X,X'))-\tilde
g_1(\nabla'^{bl}_{Z}\tilde h_1 X,X')-\tilde g_1(\tilde h_1
X,\nabla'^{bl}_{Z}X')\bigr)\\
&=2(\nabla'^{bl}_{Z}\tilde g_1)(\tilde h_1 X,X')\\
&=2(\tilde\nabla^{1 pc}_{Z}\tilde g_1)(\tilde h_1 X,X')=0.
\end{align*}
In a similar way one can prove that
$\nabla^{{\prime}{bl}}\Pi_{{\mathcal D}(-\tilde\lambda)}=0$. Next,
we check whether ${\mathcal D}(\tilde\lambda)$ and ${\mathcal
D}(-\tilde\lambda)$ are positive definite or negative definite
Legendre foliations, according to the assumptions of Theorem
\ref{legendre1}. We consider the local $g$-orthonormal bases for
${\mathcal D}(\tilde\lambda)$ and ${\mathcal D}(-\tilde\lambda)$ in
\eqref{autospazio1} and \eqref{autospazio2}, respectively. As in the
proof of Theorem \ref{principale3}, for simplifying the notation, we
put
$\beta:=\frac{1}{\sqrt{\left(1-\frac{\mu}{2}\right)-(1-\kappa)}}$.
Notice that, for any $i,j\in\left\{1,\ldots,n\right\}$, by
\eqref{parametricabis}, \eqref{formule} and
\eqref{metricacanonica1},
\begin{align*}
\tilde g_1(X_i,X_j)&=\beta\tilde g(X_i,\tilde\varphi\tilde h X_j)\\
&=-\beta\tilde g(X_i,\tilde h X_j)\\
&=-\frac{\beta}{2\sqrt{1-\kappa}}\left((2-\mu)\tilde g(X_i,\varphi h
X_j)+2(1-\kappa)\tilde g(X_i,\varphi X_j)\right)\\
&=-\frac{\beta}{2(1-\kappa)}\left(\lambda(2-\mu)+2(1-\kappa)\right)g(X_i,\varphi
h Y_j)\\
&=\beta(I_M+1)\lambda g(X_i,\varphi Y_j)\\
&=-\beta(I_M+1)\lambda\delta_{ij}.
\end{align*}
Similar computations yield $\tilde g_1(X_i,Y_j)=0$ and $\tilde
g_1(Y_i,Y_j)=\beta(I_M-1)\lambda\delta_{ij}$. Hence
\begin{align*}
\Pi_{{\mathcal
D}(\tilde\lambda)}&\left(\sqrt{\frac{I_M-1}{I_M+1}}X_i+Y_i,
\sqrt{\frac{I_M-1}{I_M+1}}X_j+Y_j\right)=2\frac{I_M-1}{I_M+1}\tilde
g_1(\tilde h_1 X_i,X_j)\\
&\quad+2\sqrt{\frac{I_M-1}{I_M+1}}\bigl(\tilde g_1(\tilde h_1
X_i,Y_j)+\tilde g_1(\tilde h_1 Y_i,X_j)\bigr)+2\tilde
g_1(\tilde h_1 Y_i,Y_j)\\
&=-\frac{2(I_M-1)}{\beta(I_M+1)}\tilde g_1(\tilde\varphi
X_i,X_j)-\frac{2}{\beta}\sqrt{\frac{I_M-1}{I_M+1}}\bigl(\tilde
g_1(\tilde\varphi X_i,Y_j)+\tilde g_1(\tilde\varphi
Y_i,X_j)\bigr)-\frac{2}{\beta}\tilde g_1(\tilde\varphi
Y_i,Y_j)\\
&=-\frac{2(I_M-1)}{\beta(I_M+1)}\tilde
g_1(X_i,X_j)-\frac{2}{\beta}\sqrt{\frac{I_M-1}{I_M+1}}\bigl(\tilde
g_1(X_i,Y_j)-\tilde g_1(Y_i,X_j)\bigr)+\frac{2}{\beta}\tilde
g_1(Y_i,Y_j)\\
&=4\lambda(I_M-1)\delta_{ij}.
\end{align*}
Arguing in the same way for ${\mathcal D}(-\tilde\lambda)$ one can
prove that
\begin{equation*}
\Pi_{{\mathcal
D}(-\tilde\lambda)}\left(-\sqrt{\frac{I_M-1}{I_M+1}}X_i+Y_i,
-\sqrt{\frac{I_M-1}{I_M+1}}X_j+Y_j\right)=4\lambda(I_M-1)\delta_{ij}.
\end{equation*}
Thus, because of the assumption $|I_M|>1$, we conclude that both
$\Pi_{{\mathcal D}(\tilde\lambda)}$ and $\Pi_{{\mathcal
D}(-\tilde\lambda)}$ are positive definite. Finally, in order to
check the last hypothesis of Theorem \ref{legendre1}, we find the
explicit  expression  of  the  Libermann   operators
$\Lambda_{{\mathcal D}(\tilde\lambda)}:TM\longrightarrow{\mathcal
D}(\tilde\lambda)$ and $\Lambda_{{\mathcal
D}(-\tilde\lambda)}:TM\longrightarrow{\mathcal D}(-\tilde\lambda)$.
Let us consider $X\in\Gamma({\mathcal D}(\tilde\lambda))$ and
$Y\in\Gamma({\mathcal D}(-\tilde\lambda))$. Then, by applying
\eqref{lambda}, $2\tilde g_1(\tilde h_1\Lambda_{{\mathcal
D}(\tilde\lambda)}Y,X)=\Pi_{{\mathcal
D}(\tilde\lambda)}(\Lambda_{{\mathcal
D}(\tilde\lambda)}Y,X)=d\eta(Y,X)=\tilde g_1(Y,\tilde\varphi_1
X)=\tilde g_1(Y,X)$, from which it follows that $2\tilde
h_1\Lambda_{{\mathcal D}(\tilde\lambda)}Y=Y$. Applying $\tilde h_1$
and since $\tilde h_1=-\frac{1}{\beta}\tilde\varphi$, we get
$\Lambda_{{\mathcal D}(\tilde\lambda)}Y=\frac{1}{2}\beta^2\tilde h_1
Y$. Thus
\begin{equation}\label{lambda1}
\Lambda_{{\mathcal D}(\tilde\lambda)}=\left\{
                                        \begin{array}{ll}
                                          0, & \hbox{on ${\mathcal D}(\tilde\lambda)\oplus\mathbb{R}\xi$,} \\
                                          \frac{1}{2\sqrt{\left(1-\frac{\mu}{2}\right)^2-(1-\kappa)}}\tilde h_1, & \hbox{on ${\mathcal D}(-\tilde\lambda)$.}
                                        \end{array}
                                      \right.
\end{equation}
In the same way one can prove that
\begin{equation}\label{lambda2}
\Lambda_{{\mathcal D}(-\tilde\lambda)}=\left\{
                                        \begin{array}{ll}
                                          - \frac{1}{2\sqrt{\left(1-\frac{\mu}{2}\right)^2-(1-\kappa)}}\tilde h_1, & \hbox{on ${\mathcal D}(\tilde\lambda)$,} \\
                                          0, & \hbox{on ${\mathcal D}(-\tilde\lambda)\oplus\mathbb{R}\xi$.}
                                        \end{array}
                                      \right.
\end{equation}
Hence, for any $Y,Y'\in\Gamma({\mathcal D}(-\tilde\lambda))$,
\begin{equation*}
\overline{\Pi}_{{\mathcal D}(\tilde\lambda)}(Y,Y')={\Pi}_{{\mathcal
D}(\tilde\lambda)}\bigl(\Lambda_{{\mathcal
D}(\tilde\lambda)}Y,\Lambda_{{\mathcal
D}(\tilde\lambda)}Y'\bigr)=\frac{\beta^4}{4}\Pi_{{\mathcal
D}(\tilde\lambda)}(\tilde h_1 Y,\tilde h_1
Y')=\frac{\beta^2}{2}\tilde g_1(Y,\tilde h_1 Y')
\end{equation*}
and for any  $X,X'\in\Gamma({\mathcal D}(\tilde\lambda))$
\begin{equation*}
\overline{\Pi}_{{\mathcal D}(-\tilde\lambda)}(X,X')={\Pi}_{{\mathcal
D}(-\tilde\lambda)}\bigl(\Lambda_{{\mathcal
D}(-\tilde\lambda)}X,\Lambda_{{\mathcal
D}(-\tilde\lambda)}X'\bigr)=\frac{\beta^4}{4}\Pi_{{\mathcal
D}(-\tilde\lambda)}(\tilde h_1 X,\tilde h_1
X')=\frac{\beta^2}{2}\tilde g_1(X,\tilde h_1 X').
\end{equation*}
On the other hand, $\Pi_{{\mathcal D}(-\tilde\lambda)}(Y,Y')=2\tilde
g_1(\tilde h_1 Y,Y')$ and $\Pi_{{\mathcal
D}(\tilde\lambda)}(X,X')=2\tilde g_1(\tilde h_1 X,X')$, so that
$\overline{\Pi}_{{\mathcal
D}(\tilde\lambda)}=\frac{4}{\beta^2}\overline{\Pi}_{{\mathcal
D}(-\tilde\lambda)}$ on ${\mathcal D}(\tilde\lambda)$ and
$\overline{\Pi}_{{\mathcal
D}(-\tilde\lambda)}=\frac{4}{\beta^2}\overline{\Pi}_{{\mathcal
D}(\tilde\lambda)}$ on ${\mathcal D}(-\tilde\lambda)$. Since the
constant $\frac{4}{\beta^2}$ is positive, we conclude that the
bi-Legendrian structure $({\mathcal D}(\tilde\lambda),{\mathcal
D}(-\tilde\lambda))$ verifies all the assumptions of Theorem
\ref{legendre1} and so, for any two positive constants $a$ and $b$
such that $ab=\frac{4}{\beta^2}$, there exists a contact metric
$(\kappa'_{a,b},\mu'_{a,b})$-structure
$(\varphi'_{a,b},\xi,\eta,g'_{a,b})$ whose associated bi-Legendrian
structure $({\mathcal D}(\tilde\lambda),{\mathcal
D}(-\tilde\lambda))$, where $\kappa'_{a,b}$ and $\mu'_{a,b}$ are
given by \eqref{costanti3}. \  Finally, notice that the Boeckx
invariant of the new contact metric
$(\kappa'_{a,b},\mu'_{a,b})$-structure
$(\varphi'_{a,b},\xi,\eta,g'_{a,b})$ is given by
$\frac{1-\frac{\mu'_{a,b}}{2}}{\sqrt{1-\kappa'_{a,b}}}=\frac{a+b}{|a-b|}$.
Hence, as $a>0$ and $b>0$, we have  $|I'_{M}|>1$ and we conclude
that $(\varphi'_{a,b},\xi,\eta,g'_{a,b})$ is of the same
classification as $(\varphi,\xi,\eta,g)$.
\end{proof}

\begin{remark}\label{osservazione}
We point out that, as it is expected, all the various contact metric
$(\kappa'_{a,b},\mu'_{a,b})$-structures in the Theorem
\ref{legendre3} induce, by means of Theorem \ref{principale1}, the
same paracontact metric $(\tilde\kappa_1,\tilde\mu_1)$-structure
$(\tilde\varphi_1,\xi,\eta,\tilde g_1)$. In other words,
$\tilde\kappa_1$ and $\tilde\mu_1$ do not depends on the arbitrarily
chosen constants $a$ and $b$ satisfying \eqref{costante}. Indeed, by
applying Theorem  \ref{principale2}, we get
$\tilde\kappa_1=\kappa'_{a,b}-2+\left(1-\frac{\mu'_{a,b}}{2}\right)^2=-1+\frac{1}{4}\left(\frac{(a+b)^2}{4}-\frac{(a-b)^2}{4}\right)=-1+\frac{ab}{4}=\kappa-2+\left(1-\frac{\mu}{2}\right)^2$
and $\tilde\mu_1=2$.
\end{remark}

Now we are able to iterate the procedure of Theorem
\ref{principale1} and Theorem \ref{principale3} and hence to define
on a contact metric $(\kappa,\mu)$-manifold $M$ a canonical sequence
of contact/paracontact metric structures as stated in the following
theorem.

\begin{theorem}\label{successione}
Let $(M,\varphi,\xi,\eta,g)$ be a contact metric
$(\kappa,\mu)$-manifold.
\begin{enumerate}
\item[(i)] If $|I_M|<1$, $M$ admits a sequence of
tensor fields $(\phi_n)_{n\in\mathbb{N}}$ and a sequence of
$(0,2)$-tensors $(G_n)_{n\in\mathbb{N}}$, defined by
\begin{gather}
\phi_0=\varphi, \ \ \  \phi_1=\frac{1}{2\sqrt{1-\kappa}}{\mathcal L}_{\xi}\phi_0,\label{primo}\\
\phi_{2n}=\frac{1}{2\sqrt{1-\kappa-\left(1-\frac{\mu}{2}\right)^2}}{\mathcal
L}_{\xi}\phi_{2n-1}, \ \ \
\phi_{2n+1}=\frac{1}{2\sqrt{1-\kappa-\left(1-\frac{\mu}{2}\right)^2}}{\mathcal
L}_{\xi}\phi_{2n},\label{secondo}\\
G_{2n}=-d\eta(\cdot,\phi_{2n})+\eta\otimes\eta, \ \ \
G_{2n+1}=d\eta(\cdot,\phi_{2n+1})+\eta\otimes\eta, \label{terzo}
\end{gather}
such that, \ for each $n\in\mathbb{N}$, \
$(\phi_{2n},\xi,\eta,G_{2n})$ \ is a contact metric
$(\kappa_{2n},\mu_{2n})$-structure and
$(\phi_{2n+1},\xi,\eta,G_{2n+1})$ is a paracontact metric
$(\kappa_{2n+1},\mu_{2n+1})$-structure, where
\begin{gather}
\kappa_0=\kappa, \ \
\kappa_{2n}=\kappa+\left(1-\frac{\mu}{2}\right)^2, \ \ \mu_{2n}=2, \label{primo1}\\
\kappa_{2n+1}=\kappa-2+\left(1-\frac{\mu}{2}\right)^2, \ \
\mu_{2n+1}=2. \label{secondo1}
\end{gather}
Moreover, for each $n\in\mathbb{N}$, $(\phi_{2n},\xi,\eta,G_{2n})$
is a Tanaka-Webster parallel structure on $M$, and
$(\phi_{2n+1},\xi,\eta,G_{2n+1})$ is the canonical paracontact
metric structure induced by $(\phi_{2n},\xi,\eta,G_{2n})$ according
to Theorem \ref{principale1}.
\item[(ii)] If $|I_M|>1$, $M$ admits a sequence of paracontact metric
structures $(\phi_n,\xi,\eta,G_n)_{n\geq 1}$, defined by
\begin{gather*}
\phi_1=\frac{1}{2\sqrt{1-\kappa}}{\mathcal L}_{\xi}\varphi, \ \
\phi_n=\frac{1}{2\sqrt{\left(1-\frac{\mu}{2}\right)^2-(1-\kappa)}}{\mathcal
L}_{\xi}\phi_{n-1}, \ \ G_{n}=d\eta(\cdot,\phi_n)+\eta\otimes\eta,
\end{gather*}
such that, for each $n\geq 1$, $(\phi_n,\xi,\eta,G_n)$ is a
paracontact metric $(\kappa_n,\mu_n)$-structure with
\begin{equation*}
\kappa_n=\kappa-2+\left(1-\frac{\mu}{2}\right)^2, \ \ \mu_n=2.
\end{equation*}
Moreover, $(\phi_1,\xi,\eta,G_1)$ is the canonical paracontact
structure induced by $(\varphi,\xi,\eta,g)$ and, for each $n\geq 2$,
$(\phi_n,\xi,\eta,G_n)$ is the canonical paracontact structure
induced by a contact metric $(\kappa'_n,\mu'_n)$-structure
$(\varphi'_n,\xi,\eta,g'_n)$ on $M$ with
\begin{equation}\label{costanti4}
\kappa'_{n}=1-\frac{(a_n-b_n)^2}{16}, \ \ \
\mu'_{n}=2-\frac{a_n+b_n}{2},
\end{equation}
$a_n$ and $b_n$ being two constants such that
\begin{equation}\label{prodotto}
a_{n}b_{n}=\frac{1}{4}\left(\left(1-\frac{\mu}{2}\right)^2-(1-\kappa)\right).
\end{equation}
\end{enumerate}
\end{theorem}
\begin{proof}
We prove the theorem arguing by induction on $n$.\\
(i) \ We distinguish the even and the odd case. The result is
trivially true for $n=0$  since $(M,\varphi,\xi,\eta,g)$ is supposed
to be a contact metric $(\kappa,\mu)$-manifold and for $n=1$ because
of Theorem \ref{principale2}. Now suppose that the assertion holds
for $(\phi_{2n},\xi,\eta,G_{2n})$, $n\geq 2$. We have to prove that
the structure $(\phi_{2n+1},\xi,\eta,G_{2n+1})$, defined by
\eqref{secondo}, is a paracontact metric
$(\kappa_{2n+1},\mu_{2n+1})$-structure, where $\kappa_{2n+1}$ and
$\mu_{2n+1}$ are given by \eqref{secondo1}. Notice that
\begin{equation*}
\phi_{2n+1}=\frac{1}{2\sqrt{1-\kappa-\left(1-\frac{\mu}{2}\right)^2}}{\mathcal
L}_{\xi}\phi_{2n}=\frac{1}{2\sqrt{1-\kappa_{2n}}}{\mathcal
L}_{\xi}\phi_{2n}
\end{equation*}
so that $(\phi_{2n+1},\xi,\eta,G_{2n+1})$ coincides with the
canonical paracontact metric structure induced on $M$ by the contact
metric $(\kappa_{2n},\mu_{2n})$-structure
$(\phi_{2n},\xi,\eta,G_{2n})$, according to Theorem
\ref{principale1}. Then, by the Theorem \ref{principale2},
$(\phi_{2n+1},\xi,\eta,G_{2n+1})$ is a paracontact metric
$(\tilde\kappa,\tilde\mu)$-structure, where
\begin{align*}
\tilde\kappa&=\kappa_{2n}-2+\left(1-\frac{\mu_{2n}}{2}\right)^2=\kappa+\left(1-\frac{\mu}{2}\right)^2-2+\left(1-\frac{2}{2}\right)^2=\kappa-2+\left(1-\frac{\mu}{2}\right)^2=\kappa_{2n+1}
\end{align*}
and $\tilde\mu=2=\mu_{2n+1}$. Now we study the odd case. Assume that
the assertion holds for $(\phi_{2n+1},\xi,\eta,G_{2n+1})$. \ We \
have \ to \ prove \ that \ $(\phi_{2n+2},\xi,\eta,G_{2n+2})$ \ is \
a \ contact \ metric \ $(\kappa_{2n+2},\mu_{2n+2})$-structure, where
$\kappa_{2n+2}$ and $\mu_{2n+2}$ are given by \eqref{primo1}. By
induction hypothesis $(\phi_{2n+1},\xi,\eta,G_{2n+1})$ is the
canonical paracontact metric structure induced by the contact metric
$(\kappa_{2n},\mu_{2n})$-structure $(\phi_{2n},\xi,\eta,G_{2n})$.
Since the Boeckx invariant of $(M,\phi_{2n},\xi,\eta,G_{2n})$ is
$0$, we can apply Theorem \ref{principale3} to the contact metric
$(\kappa_{2n},\mu_{2n})$-manifold $(M,\phi_{2n},\xi,\eta,G_{2n})$
and conclude that the paracontact metric structure
$(\phi_{2n+1},\xi,\eta,G_{2n+1})$ induces on $M$ a contact metric
structure $(\bar{\varphi}_1,\xi,\eta,\bar{g}_1)$ given by
\eqref{definzionephi} and \eqref{metrica}. Notice that
\begin{align*}
\overline{\varphi}_1&=\frac{1}{2\sqrt{1-\kappa_{2n}-\left(1-\frac{\mu_{2n}}{2}\right)^2}}{\mathcal
L}_{\xi}\phi_{2n+1}=\frac{1}{2\sqrt{1-\kappa-\left(1-\frac{\mu}{2}\right)^2-\left(1-\frac{2}{2}\right)^2}}{\mathcal
L}_{\xi}\phi_{2n+1}\\
&=\frac{1}{2\sqrt{1-\kappa-\left(1-\frac{\mu}{2}\right)^2}}{\mathcal
L}_{\xi}\phi_{2n+1}=\phi_{2n+2}.
\end{align*}
Therefore $(\phi_{2n+2},\xi,\eta,G_{2n+2})$ is a contact metric
$(\bar{\kappa}_1,\bar{\mu}_1)$-structure, where, by Theorem
\ref{principale3},
$\bar{\kappa}_1=\kappa_{2n}+\left(1-\frac{\mu_{2n}}{2}\right)^2=\kappa_{2n}=\kappa+\left(1-\frac{\mu}{2}\right)^2=\kappa_{2n+2}$
and $\bar{\mu}=2=\mu_{2n+2}$. Finally, for each $n\in\mathbb{N}$
since $\mu_{2n}=2$, applying Theorem 12 of \cite{Boeckx-08}, we
conclude that $(M,\phi_{2n},\xi,\eta,G_{2n})$ is a Tanaka-Webster
parallel space.\\
(ii) \ The result is true for $n=1$ due to Theorem \ref{principale2}
and for $n=2$ due to Theorem \ref{principale3} and Theorem
\ref{legendre3}. Now assuming that the assert holds for
$(\phi_{n},\xi,\eta,G_n)$, $n\geq 3$, we prove that it holds also
for $(\phi_{n+1},\xi,\eta,G_{n+1})$. By induction hypothesis
$(\phi_{n},\xi,\eta,G_n)$ is the canonical paracontact metric
structure induced by a contact metric $(\kappa'_n,\mu'_n)$-manifold,
 $\kappa'_n$ and $\mu'_n$ being given by \eqref{costanti4}, whose Boeckx
invariant, given by $\frac{a+b}{|a-b|}$, has absolute value strictly
greater than $1$. Hence we can apply Theorem \ref{principale3} and
conclude that $(\phi_{n},\xi,\eta,G_{n})$ induces on $M$ a
paracontact metric $(\tilde\kappa'_1,\tilde\mu'_1)$-structure
$(\tilde\varphi'_1,\xi,\eta,\tilde g'_1)$, where $\tilde\varphi'_1$,
$\tilde g'_1$ are given by \eqref{definzioneparaphi} and
\eqref{parametrica} and $\tilde\kappa'_1$,  $\tilde\mu'_1$ are given
by \eqref{costanti2}. Notice that
\begin{align*}
\tilde\varphi'_1&=\frac{1}{2\sqrt{\left(1-\frac{\mu'_n}{2}\right)^2-(1-\kappa'_n)}}{\mathcal
L}_{\xi}\phi_{n}=\frac{1}{\sqrt{\frac{(a_n+b_n)^2}{4}-\frac{(a_n-b_n)^2}{4}}}{\mathcal
L}_{\xi}\phi_{n}=\frac{1}{\sqrt{a_{n}b_{n}}}{\mathcal L}_{\xi}\phi_{n}\\
&=\frac{1}{2\sqrt{\left(1-\frac{\mu}{2}\right)^2-(1-\kappa)}}{\mathcal
L}_{\xi}\phi_{n}=\phi_{n+1}.
\end{align*}
Finally, in view of Remark \ref{osservazione}, we get
$\tilde\kappa_1=\kappa-2+\left(1-\frac{\mu}{2}\right)^2=\kappa_{n+1}$
and $\tilde\mu_{1}=2=\mu_{n+1}$.
\end{proof}

\section{Canonical Sasakian structures on contact metric $(\kappa,\mu)$-spaces}

As pointed out in Remark \ref{osservazione1}, in the proof of
Theorem \ref{principale3} we have proven that any (non-Sasakian)
contact metric $(\kappa,\mu)$-space such that  $|I_M|>1$ admits a
supplementary bi-Legendrian structure $({\mathcal
D}(\tilde\lambda),{\mathcal D}(-\tilde\lambda))$ given by the
eigendistributions of the operator $\tilde
h:=\frac{1}{4\sqrt{1-\kappa}}{\mathcal L}_{\xi}{\mathcal
L}_{\xi}\varphi$ corresponding to the eigenvalues
$\pm\tilde\lambda$, where
$\tilde\lambda:=\sqrt{\left(1-\frac{\mu}{2}\right)^2-\left(1-\kappa\right)}$.
We now prove that in fact any three of the distributions ${\mathcal
D}(\lambda)$, ${\mathcal D}(-\lambda)$, ${\mathcal
D}(\tilde\lambda)$, ${\mathcal D}(-\tilde\lambda)$ define a $3$-web
on the contact distribution of $(M,\eta)$. \ We recall that a triple
of distributions $({\mathcal D}_1,{\mathcal D}_2,{\mathcal D}_3)$ on
a smooth manifold $M$ is called  an \emph{almost $3$-web structure}
if $TM={\mathcal D}_{j}\oplus{\mathcal D}_{j}$ is satisfied for any
two different $i,j\in\left\{1,2,3\right\}$. If $\mathcal D_1$,
$\mathcal D_2$, $\mathcal D_3$ are involutive, then $({\mathcal
D}_1,{\mathcal D}_2,{\mathcal D}_3)$ is said to be simply a
\emph{$3$-web} (\cite{nagy}). \ Now, obviously one has that
${\mathcal D}={\mathcal D}(\lambda)\oplus{\mathcal D}(-\lambda)$ and
${\mathcal D}={\mathcal D}(\tilde\lambda)\oplus{\mathcal
D}(-\tilde\lambda)$, so that it is sufficient to prove that
${\mathcal D}={\mathcal D}(\lambda)\oplus{\mathcal
D}(\tilde\lambda)$, ${\mathcal D}={\mathcal
D}(\lambda)\oplus{\mathcal D}(-\tilde\lambda)$, ${\mathcal
D}={\mathcal D}(-\lambda)\oplus{\mathcal D}(\tilde\lambda)$ and
${\mathcal D}={\mathcal D}(-\lambda)\oplus{\mathcal
D}(-\tilde\lambda)$. Let $\{X_1,\ldots,X_n,Y_1:=\varphi
X_1,\ldots,Y_n:=\varphi X_n,\xi\}$ be a (local) orthonormal
$\varphi$-basis of eigenvectors of $h$. Then ${\mathcal
D}(\lambda)=\textrm{span}\{X_1,\ldots,X_n\}$, ${\mathcal
D}(-\lambda)=\textrm{span}\{Y_1,\ldots,Y_n\}$ and ${\mathcal
D}(\tilde\lambda)$, ${\mathcal D}(-\tilde\lambda)$ are given,
respectively, by \eqref{autospazio1}, \eqref{autospazio2}. Using
these local expressions, by some elementary arguments of linear
algebra, it easily follows that, putting
$\gamma:=\sqrt{\frac{I_M-1}{I_M+1}}$,
\begin{gather*}
\left\{X_1,\ldots,X_n,\gamma X_1+Y_1,\ldots,\gamma X_n+Y_n\right\},\\
\left\{X_1,\ldots,X_n,-\gamma X_1+Y_1,\ldots,-\gamma X_n+Y_n\right\},\\
\left\{Y_1,\ldots,Y_n,\gamma X_1+Y_1,\ldots,\gamma X_n+Y_n\right\},\\
\left\{Y_1,\ldots,Y_n,-\gamma X_1+Y_1,\ldots,-\gamma
X_n+Y_n\right\},
\end{gather*}
are all local bases of the contact distribution $\mathcal D$. Hence
the assertion follows.

As shown in \cite{marchiafava}, to any almost $3$-web one can
associate a canonical almost anti-hypercomplex structure, that is a
triple $(I_1,I_2,I_3)$ consisting of an almost complex structure
$I_{1}$ and two anti-commuting almost product structures $I_2$,
$I_3$ satisfying $I_{2}I_{3}=I_{1}$ (and hence
$I_{2}I_{1}=-I_{1}I_{2}=I_{3}$, $I_{1}I_{3}=-I_{3}I_{1}=I_{2}$).
Conversely, any almost anti-hypercomplex structure determines four
almost $3$-webs given by the eigendistributions of $I_2$ and $I_3$
corresponding to the eigenvalues $\pm 1$. \ Consequently, any
contact metric $(\kappa,\mu)$-manifold such that $|I_M|>1$ admits a
canonical anti-hypercomplex structure on the contact distribution
via the above $3$-webs. Such anti-hypercomplex structure is in fact
given by $(\bar\varphi_{-}|_{\mathcal D},\tilde\varphi|_{\mathcal
D},\tilde\varphi_{1}|_{\mathcal D})$ in the case $I_M<-1$ and by
$(\bar\varphi_{+}|_{\mathcal D},\tilde\varphi_{1}|_{\mathcal
D},\tilde\varphi|_{\mathcal D})$ in the case $I_M>1$, where
$\tilde\varphi$, $\tilde\varphi_1$ are given, respectively, by
\eqref{paracanonical}, \eqref{definzioneparaphi}, and
\begin{gather*}
\bar\varphi_{\pm}:=\pm\frac{1}{\sqrt{\left(1-\frac{\mu}{2}\right)^2-\left(1-\kappa\right)}}\left(\left(1-\frac{\mu}{2}\right)\varphi+\varphi
h\right).
\end{gather*}
Indeed using \eqref{paracanonical}, \eqref{definzioneparaphi} and
the relations $h^2=(\kappa-1)\varphi^2$, $\varphi h=-h\varphi$, one
can easily check by a straightforward computation that
$\tilde\varphi$ and $\tilde\varphi_1$ induce two anti-commuting
almost product structures on $\mathcal D$ and that
$\tilde\varphi\tilde\varphi_{1}=\bar\varphi_{-}$ and
$\tilde\varphi_{1}\tilde\varphi=\bar\varphi_{+}$. We prove that
$\bar\varphi_{-}$ and $\bar\varphi_{+}$ are almost contact
structures compatible with $\eta$. Indeed
\begin{align*}
{\bar\varphi_{-}}^2&=\frac{1}{\left(1-\frac{\mu}{2}\right)^2-\left(1-\kappa\right)}\left(\left(1-\frac{\mu}{2}\right)^2\varphi^2+\varphi
h\varphi h +\left(1-\frac{\mu}{2}\right)\varphi^2
h+\left(1-\frac{\mu}{2}\right)\varphi h\varphi\right)\\
&=\frac{1}{\left(1-\frac{\mu}{2}\right)^2-\left(1-\kappa\right)}\left(\left(1-\frac{\mu}{2}\right)^2\varphi^2-\varphi^2
h^2\right)\\
&=\frac{1}{\left(1-\frac{\mu}{2}\right)^2-\left(1-\kappa\right)}\left(\left(1-\frac{\mu}{2}\right)^2\varphi^2-\left(1-\kappa\right)\varphi^2\right)\\
&=\varphi^2\\
&=-I+\eta\otimes\xi.
\end{align*}
Analogously one can prove that
${\bar\varphi_{+}}^2=-I+\eta\otimes\xi$. Moreover, for each almost
contact structure $(\bar\varphi_{-},\xi,\eta)$ and
$(\bar\varphi_{+},\xi,\eta)$ one can define an associated metric
${\bar g}_{-}$ and ${\bar g}_{+}$, respectively, by
\begin{equation}\label{Sasakian}
{\bar g}_{\pm}(X,Y)=-d\eta(X,\bar\varphi_{\pm}Y)+\eta(X)\eta(Y).
\end{equation}
We prove that $\bar g_{-}$ (respectively, $\bar g_{+}$) is a
Riemannian metric compatible with the almost contact structure
$(\bar\varphi_{-},\xi,\eta)$ (respectively,
$(\bar\varphi_{+},\xi,\eta)$). By \eqref{Sasakian} straightforwardly
follows that $\bar g_{-}$ is  non-degenerate, symmetric  and
satisfies ${\bar g}_{-}(\bar\varphi_{-}X,\bar\varphi_{-}Y)={\bar
g}_{-}(X,Y)-\eta(X)\eta(Y)$. We prove that it positive definite. By
\eqref{Sasakian} we have that ${\bar g}_{-}(\xi,\xi)=1$, so that it
is sufficient to prove that ${\bar g}_{-}(X,X)>0$ for any
$X\in\Gamma({\mathcal D})$, $X\neq 0$. We decompose $X$ is its
components $X_{\lambda}$ and $X_{-\lambda}$ according to the
decomposition ${\mathcal D}={\mathcal D}(\lambda)\oplus{\mathcal
D}(-\lambda)$. For simplifying  the notation, as in $\S$
\ref{sezioneprincipale}, we put
$\beta:=\frac{1}{\sqrt{\left(1-\frac{\mu}{2}\right)^2-\left(1-\kappa\right)}}$.
Then we have
\begin{align*}
{\bar
g}_{-}(X,X)&=\beta\left(\left(1-\frac{\mu}{2}\right)d\eta(X,\varphi
X)+d\eta(X,\varphi h X)\right)\\
&=-\beta\left(\left(1-\frac{\mu}{2}\right)g(X,X)+g(X,h
X)\right)\\
&=-\beta\left(\left(1-\frac{\mu}{2}\right)\left(g(X_{\lambda},X_{\lambda})+g(X_{-\lambda},X_{-\lambda})\right)+\lambda
g(X_{\lambda},X_{\lambda})-\lambda
g(X_{-\lambda},X_{-\lambda})\right)\\
&=-\beta\left(\left(1-\frac{\mu}{2}+\sqrt{1-\kappa}\right)g(X_{\lambda},X_{\lambda})+\left(1-\frac{\mu}{2}-\sqrt{1-\kappa}\right)g(X_{-\lambda},X_{-\lambda})\right).
\end{align*}
Since we are assuming $I_M<-1$, we have that
$1-\frac{\mu}{2}+\sqrt{1-\kappa}<0$ and
$1-\frac{\mu}{2}-\sqrt{1-\kappa}<0$, so that ${\bar g}_{-}(X,X)>0$.
Analogous arguments work for $\bar g_{+}$, where one uses the
assumption $I_M>1$. Finally, directly from \eqref{Sasakian} it
follows that $d\eta(\cdot,\cdot)=\bar
g_{\pm}(\cdot,\bar\varphi_{\pm})$, and we conclude that
$(\bar\varphi_{-},\xi,\eta,\bar g_{-})$ and
$(\bar\varphi_{+},\xi,\eta,\bar g_{+})$ are contact metric
structures. We prove that they are in fact Sasakian structures. We
argue on $(\bar\varphi_{-},\xi,\eta,\bar g_{-})$, since the same
arguments work also for $(\bar\varphi_{+},\xi,\eta,\bar g_{+})$. We
firstly prove that the contact metric structure is $K$-contact, i.e.
 the tensor field ${\bar h}_{-}:=\frac{1}{2}{\mathcal
L}_{\xi}\bar\varphi_{-}$ vanishes identically. Indeed, by using
\eqref{formule}, we have
\begin{align*}
2{\bar h}_{-}&=-\beta\left(\left(1-\frac{\mu}{2}\right){\mathcal
L}_{\xi}\varphi+{\mathcal L}_{\xi}(\varphi h)\right)\\
&=-\beta\left(\left(1-\frac{\mu}{2}\right){\mathcal
L}_{\xi}\varphi+({\mathcal L}_{\xi}\varphi)\circ h + \varphi\circ({\mathcal L}_{\xi}h)\right)\\
&=-\beta\left((2-\mu)h+2h^2 + (2-\mu)\varphi^2 h +
2(1-\kappa)\varphi^2\right)=0.
\end{align*}
Now we preliminarily observe that $\bar\varphi_{-}{\mathcal
D}(\lambda)={\mathcal D}(-\lambda)$ and $\bar\varphi_{-}{\mathcal
D}(-\lambda)={\mathcal D}(\lambda)$. Thus the Legendre foliations
${\mathcal D}(\lambda)$, ${\mathcal D}(-\lambda)$ are conjugate with
respect to $\bar\varphi_{-}$, and consequently they are mutually
orthogonal with respect to ${\bar g}_{-}$. Then we can apply Theorem
\ref{lemmarocky}. Note that
$\nabla^{bl}\bar\varphi_{-}=-\beta\left(\left(1-\frac{\mu}{2}\right)\nabla^{bl}\varphi+\nabla^{bl}(\varphi
h)\right)=0$, since $\nabla^{bl}\varphi=\nabla^{bl}h=0$. Hence, by
Theorem \ref{lemmarocky}, we have that
$\nabla^{bl}_{X}X'=-(\bar\varphi_{-}[X,\bar\varphi_{-}X'])_{{\mathcal
D}(\lambda)}$ for all $X,X'\in\Gamma({\mathcal D}(\lambda))$. Hence
\begin{align*}
(N_{\bar\varphi_{-}}(X,X'))_{{\mathcal
D}(\lambda)}&=-[X,X']-(\bar\varphi_{-}[\bar\varphi_{-}X,X'])_{{\mathcal D}(\lambda)}-(\bar\varphi_{-}[X,\bar\varphi_{-}X'])_{{\mathcal D}(\lambda)}\\
&=-[X,X']-\nabla^{bl}_{X'}X+\nabla^{bl}_{X}X'\\
&=T^{bl}(X,X')\\
&=2d\eta(X,X')\xi=0.
\end{align*}
Analogously, $(N_{\bar\varphi_{-}}(Y,Y'))_{{\mathcal
D}(-\lambda)}=0$ for all $Y,Y'\in\Gamma({\mathcal D}(-\lambda))$.
Now, for all $X,X'\in\Gamma({\mathcal D}(\lambda))$,
\begin{align*}
N_{\bar\varphi_{-}}(\bar\varphi_{-}X,\bar\varphi_{-}X')&=-[\bar\varphi_{-}X,\bar\varphi_{-}X']+[{\bar\varphi_{-}}^2
X,{\bar\varphi_{-}}^2 X']-\bar\varphi_{-}[{\bar\varphi_{-}}^2
X,\bar\varphi_{-}X']-\bar\varphi_{-}[\bar\varphi_{-}X,{\bar\varphi_{-}}^2 X']\\
&=-[\bar\varphi_{-} X,\bar\varphi_{-}
X']+[X,X']+\bar\varphi_{-}[X,\bar\varphi_{-}
X']+\bar\varphi_{-}[\bar\varphi_{-} X,X']\\
&=-N_{\bar\varphi_{-}}(X,X'),
\end{align*}
hence  $(N_{\bar\varphi_{-}}(X,X'))_{{\mathcal
D}(-\lambda)}=-(N_{\bar\varphi_{-}}(\bar\varphi_{-}
X,\bar\varphi_{-} X'))_{{\mathcal D}(-\lambda)}=0$. Next, by
\eqref{formulenijenhuis2}, $N_{\bar\varphi_{-}}(X,X')$ has zero
component also in the direction of $\xi$, so  we conclude that
$N_{\bar\varphi_{-}}(X,X')=0$. In the same way one can show that
$N_{\bar\varphi_{-}}(Y,Y')=0$ for all $Y,Y'\in\Gamma({\mathcal
D}(-\lambda))$. Moreover, \eqref{formulenijenhuis1} implies that
$N_{\bar\varphi_{-}}(X,Y)=0$ for all $X\in\Gamma({\mathcal
D}(\lambda))$ and $Y\in\Gamma({\mathcal D}(-\lambda))$. Finally,
directly by \eqref{nijenhuis} we have
$\eta(N_{\bar\varphi_{-}}(Z,\xi))=0$ for all $Z\in\Gamma({\mathcal
D})$, and from \eqref{formulenijenhuis1} it follows  that
${{\bar\varphi}_{-}}(N_{\bar\varphi_{-}}(Z,\xi))=0$. Hence
$N_{\bar\varphi_{-}}(Z,\xi)\in\ker(\eta)\cap\ker(\bar\varphi_{-})=\{0\}$.
Thus the tensor field $N_{\bar\varphi_{-}}$ vanishes identically and
so $(\bar\varphi_{-},\xi,\eta,\bar g_{-})$ is a Sasakian structure.
In the same way one argues for $(\bar\varphi_{+},\xi,\eta,\bar
g_{+})$.

In conclusion we have proven the following result.

\begin{theorem}\label{principale4}
Let $(M,\varphi,\xi,\eta,g)$ be a non-Sasakian contact metric
$(\kappa,\mu)$-space such that $|I_M|>1$. Then $(M,\eta)$ admits a
compatible Sasakian structure $(\bar\varphi_{-},\xi,\eta,\bar
g_{-})$ or $(\bar\varphi_{+},\xi,\eta,\bar g_{+})$, according to the
fact that $I_M<-1$ or $I_M>1$, respectively, where
\begin{equation*}
\bar\varphi_{\pm}:=\pm\frac{1}{\sqrt{\left(1-\frac{\mu}{2}\right)^2-\left(1-\kappa\right)}}\left(\left(1-\frac{\mu}{2}\right)\varphi+\varphi
h\right), \ \ \bar
g_{\pm}:=-d\eta(\cdot,\bar\varphi_{\pm}\cdot)+\eta\otimes\eta.
\end{equation*}
Furthermore, the triple
$(\bar\varphi_{-},\tilde\varphi,\tilde\varphi_{1})$ in the case
$I_M<-1$, or $(\bar\varphi_{+},\tilde\varphi_{1},\tilde\varphi)$ in
the case $I_M>1$,   induces an almost anti-hypercomplex structure on
the contact distribution of $(M,\eta)$, where $\tilde\varphi$,
$\tilde\varphi_1$ are given, respectively, by \eqref{paracanonical},
\eqref{definzioneparaphi}.
\end{theorem}

\begin{remark}
Theorem \ref{principale4} should be compared with Corollary 3.7 in
\cite{Mino-sub}, where a similar result has been found, but using
completely different methods and where, however, the explicit
expression of the Sasakian structure was not given.
\end{remark}

\begin{remark}
In view of Corollary \ref{tanakaparallel} and Theorem
\ref{principale4} it appears that a possible geometric
interpretation of the Boeckx invariant $I_M$ is related to the
existence on the manifold of compatible  Tanaka-Webster parallel
structures or Sasakian structures, according to have $|I_M|<1$ or
$|I_M|>1$, respectively. Whereas not much one can say about those
contact metric $(\kappa,\mu)$-spaces such that $I_M=\pm 1$, which
seem to have a completely different geometric behavior and so
deserve to be studied in some subsequent paper.
\end{remark}


\small

\begin{thebibliography}{99}
\bibitem{blair-goldberg} D.~E.~Blair, S.~I.~Goldberg, \textit{Topology of almost contact
manifolds}, J. Differential Geom. \textbf{1} (1967), 347--354.

\bibitem{blair-1} D.~E.~Blair,  \textit{Two remarks on contact metric
structures}, T\^{o}hoku Math. J. \textbf{28} (1976), 373--379.

\bibitem{Blair-02} D.~E.~Blair, \textit{Riemannian geometry of
contact and symplectic manifolds}, Progress in Mathematics, 203.
Birkh\"{a}user Boston, Inc., Boston, MA, 2002.

\bibitem{BKP-95} D.~E.~Blair, T.~Koufogiorgos, B.~J.~Papantoniou,
\textit{Contact metric manifolds satisfyng a nullity condition},
Israel J. Math. \textbf{91} (1995), 189--214.


\bibitem{Boeckx-00} E.~Boeckx, \textit{A full classification of
contact metric $(\kappa,\mu )$-spaces}, Illinois J. Math.
\textbf{44} (2000), 212--219.

\bibitem{Boeckx-06} E.~Boeckx, \textit{Contact-homogeneous locally $\varphi$-symmetric manifolds},
Glasgow Math. J. \textbf{48} (2006), 93--109.

\bibitem{Boeckx-08} E.~Boeckx, J. T. Cho, \textit{Pseudo-Hermitian symmetries}, Israel J. Math. \textbf{166} (2008), 125--145.

\bibitem{Mino-05} B.~Cappelletti Montano, \textit{Bi-Legendrian
connections}, Ann. Polon. Math. \textbf{86} (2005), 79--95.

\bibitem{Mino-07} B.~Cappelletti Montano, \textit{Some remarks on the
generalized Tanaka-Webster connection of a contact metric manifold},
Rocky Mountain J. Math., to appear.

\bibitem{Mino-sub} B.~Cappelletti Montano, \textit{The foliated structure of  contact metric
$(\kappa,\mu)$-spaces}, Illinois J. Math., to appear.

\bibitem{Mino-Luigia-07} B.~Cappelletti Montano, L.~Di Terlizzi,
\textit{Contact metric $(\kappa,\mu)$-spaces as bi-Legendrian
manifolds}, Bull. Austral. Math. Soc. \textbf{77} (2008), 373--386.

\bibitem{Mino-Luigia-Tripathi-08} B.~Cappelletti Montano, L.~Di
Terlizzi, M.~M.~Tripathi, \textit{Invariant submanifolds of contact
$(\kappa,\mu)$-manifolds}, Glasgow Math. J. \textbf{50} (2008),
499--507.

\bibitem{Mino-08} B.~Cappelletti Montano, \textit{Bi-Legendrian structures and paracontact
geometry}, Int. J. Geom. Meth. Mod. Phys. \textbf{6} (2009),
487--504.

\bibitem{kaneyuki1} S.~Kaneyuki, F.~L.~Williams, \textit{Almost paracontact and parahodge structures on
manifolds}, Nagoya Math. J. \textbf{99} (1985), 173--187.

\bibitem{kobayashi1}S. Kobayashi, K. Nomizu,  \textit{Foundations of differential geometry, Vol. I}, Interscience Publishers, 1963.

\bibitem{koufogiorgos} T.~Koufogiorgos, M.~Markellos,
B.~J.~Papantoniou, \textit{The harmonicity of the Reeb vector field
on contact metric $3$-manifolds}, Pacific J. Math. \textbf{234}
(2008), 325--344.

\bibitem{libermann} P.~Libermann, \textit{Legendre foliations on
contact manifolds}, Different. Geom. Appl. \textbf{1} (1991),
57--76.

\bibitem{marchiafava} S.~Marchiafava, P.~T.~Nagy, \textit{(Anti-)hypercomplex structures and 3-webs on a
manifold}, Report n. 38 of the Department of Mathematics ``G.
Castelnuovo'', University ``La Sapienza'' of Rome, 2003.

\bibitem{nagy} P.~T.~Nagy, \textit{Invariant tensorfields and the canonical connection of a
3-web}, Aequationes Math. \textbf{35} (1988), 31--44.

\bibitem{pang} M.~Y.~Pang, \textit{The structure of Legendre
foliations}, Trans. Amer. Math. Soc. \textbf{320} n. 2 (1990),
417--453.

\bibitem{tanno} S.~Tanno, \textit{Variational problems on contact Riemannian
manifolds}, Trans. Amer. Math. Soc. \textbf{314} (1989), 349--379.

\bibitem{zamkovoy} S.~Zamkovoy, \textit{Canonical connections on paracontact
manifolds}, Ann. Glob. Anal. Geom. \textbf{36} (2009), 37--60.
\end{thebibliography}
\end{document}